\documentclass[a4paper,11pt, oneside]{amsart}

\usepackage[cp1250]{inputenc}
\usepackage{amsmath}
\usepackage{amssymb,amsbsy,amsmath,amsfonts,amssymb,amscd}
\usepackage{latexsym}
\usepackage{amsthm}
\usepackage{mathrsfs}
\usepackage{graphics}
\usepackage{color}
\usepackage{tikz}
\usetikzlibrary{arrows}
\usepackage[new]{old-arrows}
\usepackage{tikz-cd}
\usepackage{pdfpages} 
\usepackage{longtable}
\usepackage{xy}
\usepackage{chngcntr}
\usepackage{array}
\usepackage{color}
\usepackage{comment}
\usepackage[shortlabels]{enumitem}
\usepackage{geometry}
\usepackage{setspace}
\usepackage{multirow}
\usepackage[pagebackref=true]{hyperref}
\usepackage{textcmds} 
\usepackage{pgf, tikz}
\usetikzlibrary{3d,angles,quotes,calc,shapes,decorations,positioning,intersections,through,angles, patterns, babel, trees, mindmap, tikzmark} 
\usepackage{subcaption}
\usepackage{caption}
\usepackage[alphabetic,backrefs,lite]{amsrefs} 

\input xy
\xyoption{all}

\newcommand\sO{{\mathcal O}}

\newcommand\sA{{\mathcal A}}

\newcommand\sN{{\mathcal N}}
\newcommand\sK{{\mathcal K}}

\newcommand{\CC}{\ensuremath{\mathbb{C}}}
\newcommand{\RR}{\ensuremath{\mathbb{R}}}
\newcommand{\ZZ}{\ensuremath{\mathbb{Z}}}
\newcommand{\QQ}{\ensuremath{\mathbb{Q}}}

\newcommand\Lam{\Lambda}

\newcommand\ze{\zeta}

\DeclareMathOperator{\Fix}{Fix}
\DeclareMathOperator{\Ext}{Ext}

\DeclareMathOperator{\Sym}{Sym}

\DeclareMathOperator{\Stab}{Stab}

\DeclareMathOperator{\AGL}{AGL}

\DeclareMathOperator{\Aut}{Aut}

\DeclareMathOperator{\diag}{diag}

\DeclareMathOperator{\ord}{ord}

\DeclareMathOperator{\id}{id}

\DeclareMathOperator{\He}{He}
\DeclareMathOperator{\diff}{diff}

\DeclareMathOperator{\im}{im}

\DeclareMathOperator{\GL}{GL}

\DeclareMathOperator{\Bihol}{Bihol}

\DeclareMathOperator{\SL}{SL}
\DeclareMathOperator{\mult}{mult}
\DeclareMathOperator{\Eig}{Eig}
\DeclareMathOperator{\lcm}{lcm}
\DeclareMathOperator{\cone}{cone}
\DeclareMathOperator{\Conv}{Conv}

\DeclareMathOperator{\Ort}{O}

\DeclareMathOperator{\fix}{fix}
\DeclareMathOperator{\orb}{orb}
\DeclareMathOperator{\con}{conj}
\DeclareMathOperator{\Spec}{Spec}
\DeclareMathOperator{\ter}{ter}

\newcommand{\bigslant}[2]{{\raisebox{.2em}{$#1$}/\raisebox{-.2em}{$#2$}}}

\def\eea{\end{eqnarray*}}
\def\bea{\begin{eqnarray*}}

\newcommand\dual{\mathrel{\raise3pt\hbox{$\underline{\mathrm{\thinspace d
\thinspace}}$}}}
\newcommand\qe{\ifhmode\unskip\nobreak\fi\quad $\Box$}       

\def\BOX{\hfill\lower.5\baselineskip\hbox{$\Box$}}

\DeclareMathOperator{\SheafHom}{\mathscr{H}\text{\kern -4pt {\textit{om}}}\,}
\DeclareMathOperator{\SheafExt}{\mathscr{E}\text{\kern -3pt {\textit{xt}}}\,}

\newtheorem{theorem}{Theorem}
\newtheorem{theo}[theorem]{Theorem}
\newtheorem{Theorem}[theorem]{Theorem}

\newtheorem{prop}[theorem]{Proposition}
\newtheorem{Proposition}[theorem]{Proposition}
\newtheorem{cor}[theorem]{Corollary}
\newtheorem{Corollary}[theorem]{Corollary}
\newtheorem{lemma}[theorem]{Lemma}
\newtheorem{Lemma}[theorem]{Lemma}
\newtheorem{example}[theorem]{Example}

\numberwithin{theorem}{section}
\numberwithin{equation}{section}

\theoremstyle{definition}
\newtheorem{Remark}[theorem]{Remark}
\newtheorem{remark}[theorem]{Remark}
\newenvironment{rem}{\begin{remark}\rm}{\end{remark}}
\newtheorem{defin}[theorem]{Definition}
\newtheorem{notation}[theorem]{Notation}
\newtheorem{Notation}[theorem]{Notation}
\newenvironment{definition}{\begin{defin}\rm}{\end{defin}}

\makeatletter
\def\tagform@#1{\maketag@@@{\ignorespaces#1\unskip\@@italiccorr}}
\makeatother

\newcolumntype{H}{@{}>{\lrbox0}l<{\endlrbox}} 
\newcommand{\mylabel}[2]{#2\def\@currentlabel{#2}\label{#1}}

\hypersetup{colorlinks=true, unicode=true, linkcolor=[rgb]{0.10,0.05,0.67}, citecolor=[rgb]{0.10,0.05,0.67}, filecolor=[rgb]{0.10,0.05,0.67}, urlcolor=[rgb]{0.10,0.05,0.67}}

\begin{document}

\title[Rigid Canonical Torus Quotients]{The Classification of Rigid Torus Quotients with Canonical Singularities in Dimension Three}
\author{ Christian Glei\ss ner and Julia Kotonski}
\address{Christian Gleissner and Julia Kotonski  \newline University of Bayreuth, Universit\"atsstr. 30, D-95447 Bayreuth, Germany}
\email{christian.gleissner@uni-bayreuth.de, julia.kotonski@uni-bayreuth.de}

\thanks{
\textit{Acknowledgements:} The authors would like to thank Ingrid Bauer and Fabrizio Catanese for useful comments and discussions.
}

\begin{abstract}
    We provide a fine classification of rigid three-dimensional torus quotients with isolated canonical singularities, up to biholomorphism and diffeomorphism. This complements the classification of Calabi-Yau 3-folds of type $\rm{III}_0$, which are those quotients with Gorenstein singularities.
\end{abstract}

\maketitle

\tableofcontents

\section{Introduction}

A \textit{generalized hyperelliptic manifold} is a quotient of a complex torus by a free action of a non-trivial finite group which does not contain translations. The investigation of these manifolds is a classical topic dating back to the beginning of the 20th century, where Bagnera and de Franchis, as well as Enriques and Severi, gave a complete classification in the surface case (cf. \cite{BdF}, \cite{Enriques}).
For their achievements, they were awarded with the  Bordin prize in 1907 and in 1909, respectively.  

Later on in the 1970s, Ushida and Yoshihara provided a complete list of all finite groups leading to hyperelliptic 3-folds (cf. \cite{uchida}). This list consists of 
finitely many abelian groups and the dihedral group $\mathcal D_4$ of order $8$. 
In the cases where the group $G$ is abelian, the quotients were classified by Lange \cite{lange}.
In contrast to the surface case, there exist hyperelliptic 3-folds which are Calabi-Yau. They were studied in \cite{OguisoQuotientType} by Oguiso and Sakurai, independently of the work of the above authors.  They showed that their holonomy group is either $\ZZ_2^2$ or $\mathcal D_4$ and gave explicit examples of 3- and 2-dimensional families, respectively.
Note that examples with group $\mathcal D_4$ were also discovered in \cite{SzczD4} and \cite{Johnson} as examples of 6-dimensional flat Riemannian manifolds possessing a complex structure. In \cite{CatDemD4}, Catanese and Demleitner gave an entire classification of hyperelliptic 3-folds with group $\mathcal D_4$.

It was observed in \cite{DG}*{Theorem~1.1} that all hyperelliptic 3-folds have non-trivial deformations. However, relaxing the freeness-condition and allowing isolated canonical singularities, many rigid examples in dimension 3 arise (for examples, see \cite{beauville}, \cite{BG21}). This is where the present work comes in: the aim of this paper is the full classification of all of these quotients up to biholomorphism and homeomorphism.

Let us give a brief overview of the partial results which are known. The rigidity of the action immediately implies that there are no global $G$-invariant holomorphic 1- and 2-forms, hence the irregularities $q_1$ and $q_2$ of the quotient $X$ vanish. If the volume form of the torus is preserved under the group action, then the singularities are all Gorenstein and $X$ admits a crepant resolution $f\colon \hat{X}\to X$ and $\hat{X}$ is then a Calabi-Yau 3-fold. In the above mentioned paper, Oguiso and Sakurai characterized the pairs $(\hat{X},f)$ as a special class of so-called $c_2$-contractions. Moreover, they showed that the group $G$ is cyclic of order $3$ or $7$, or one of the groups
	\[\He(3)=\langle g,h,k\:\mid\: g^3=h^3=k^3=[g,k]=[h,k]=1,\:[g,h]=k\rangle,\quad \ZZ_3^2=\langle h, k\rangle<\He(3).\]
They also described the linear parts of the actions, which turned out to be unique for each group up to equivalence of representations and automorphisms. However, different choices of the translation part of the action may lead to different biholomorphism or even homeomorphism classes of quotients. Motivated by this observation,  we established a fine classification in \cite{GK}*{Theorem~1.1}:\\
There are exactly eight biholomorphism classes of rigid quotients $T/G$ of 3-dimensional complex tori by a holomorphic action with finite fixed locus that preserves the volume form on $T$. They are pairwise topologically distinct. Table~\ref{tab:CalabiYau} contains precisely one representative $Z_i$ for each class. Furthermore, the crepant Calabi-Yau resolutions of the singular quotients are still rigid.
	\medskip
 \begin{center}
    \begin{table}[h]
		{\footnotesize
		\bgroup\def\arraystretch{1.5}\begin{tabular}{|c|c|c|l|c|c|} \hline 
		$i$	&\emph{$G$} & \emph{ $\Lambda$} & \multicolumn{1}{c|}{action }& singularities & $\pi_1(Z_i)$ \\ \hline \hline
		1	&$\ZZ_7$ & $\Lam(\ze_7,\ze_7^2,\ze_7^4) $ & \begin{tabular}{l}$\Phi(1)(z)=\diag(\ze_7,\ze_7^2,\ze_7^4)\cdot z$\end{tabular}& $7\times \tfrac{1}{7}(1,2,4)$ & $\{1\}$ \\ \hline \hline
		2	&$\ZZ_3$ & $\ZZ[\ze_3]^3$ & \begin{tabular}{l}$\Phi(1)(z)=\diag
		(\ze_3,\ze_3,\ze_3)\cdot z $\end{tabular} & $27\times \tfrac{1}{3}(1,1,1)$ & $\{1\}$\\ \hline \hline
		3	&$\ZZ_3^2$ & $\ZZ[\ze_3]^3$ & \begin{tabular}{l}$\Phi(h)(z)=\diag(1,\ze_3^2,\ze_3)\cdot z+(t,t,t)$\\  $\Phi(k)(z)=\diag(\ze_3,\ze_3,\ze_3)\cdot z$\end{tabular}& $9\times \tfrac{1}{3}(1,1,1)$ & $\ZZ_3$\\ \hline
		4   & $\ZZ_3^2$ & $\ZZ[\ze_3]^3+\ZZ(t,t,0) $ & \begin{tabular}{l}$\Phi(h)(z)=\diag(1,\ze_3^2,\ze_3)\cdot z+\tfrac{1}{3}(1,1,3t)$\\  $\Phi(k)(z)=\diag(\ze_3,\ze_3,\ze_3)\cdot z$\end{tabular} &  $9\times \tfrac{1}{3}(1,1,1)$ & $\ZZ_3$\\ \hline
		5   & $\ZZ_3^2$ & $\ZZ[\ze_3]^3+\ZZ(t,t,t) $ & \begin{tabular}{l}$\Phi(h)(z)=\diag(1,\ze_3^2,\ze_3)\cdot z+\tfrac{1}{3}(1,1,1)$\\  $\Phi(k)(z)=\diag(\ze_3,\ze_3,\ze_3)\cdot z$\end{tabular} &  $9\times \tfrac{1}{3}(1,1,1)$ & $\ZZ_3$ \\ \hline
		6   & $\ZZ_3^2$ & $\ZZ[\ze_3]^3+\ZZ(t,t,t)+\ZZ(t,-t,0) $ & \begin{tabular}{l}$\Phi(h)(z)=\diag(1,\ze_3^2,\ze_3)\cdot z+\tfrac{1}{3}(1,1,1)$\\  $\Phi(k)(z)=\diag(\ze_3,\ze_3,\ze_3)\cdot z$\end{tabular} &  $9\times \tfrac{1}{3}(1,1,1)$ & $\ZZ_3$\\
		 \hline \hline
		7 &$\He(3)$ & $\ZZ[\ze_3]^3+\ZZ(t,t,t)$ & \begin{tabular}{l} $\Phi(g)(z)=\begin{pmatrix}
				0 & 0 & 1\\ 1 & 0 & 0 \\ 0 & 1 & 0
			\end{pmatrix} \cdot z+ (t,0,0) $ \\ $\Phi(h)(z)=\diag(1,\ze_3^2,\ze_3)\cdot z+ \tfrac{2}{3}(1,1,1)$ \end{tabular} & $3\times \tfrac{1}{3}(1,1,1)$ & $\ZZ_3^2$\\ \hline
		8 & $\He(3)$ & $\ZZ[\ze_3]^3+\ZZ(t,t,t)+\ZZ(t,-t,0)$ & \begin{tabular}{l} $\Phi(g)(z)=\begin{pmatrix}
				0 & 0 & 1\\ 1 & 0 & 0 \\ 0 & 1 & 0
			\end{pmatrix}\cdot z + (t,0,0) $ \\ $\Phi(h)(z)=\diag(1,\ze_3^2,\ze_3)\cdot z+ \tfrac{2}{3}(1,1,1)$ \end{tabular} & $3\times \tfrac{1}{3}(1,1,1)$ & $\ZZ_3^2$\\ \hline
		\end{tabular}\egroup}
   \caption{Calabi-Yau quotients. In the table, $t:=(1+2\ze_3)/3$ and $\Lam(\ze_7,\ze_7^2,\ze_7^4)$ has the basis $\{(\ze_7^k,\ze_7^{2k},\ze_7^{4k})\mid k=0,\ldots,5\}.$}
   \label{tab:CalabiYau}	
    \end{table}
    \end{center}

It remains to analyze the complementary case where the volume form of the torus is not preserved, or equivalently, the geometric genus $p_g$ of the quotient is $0$. The main result of this paper is the classification of the groups and the quotients in this case:

\begin{Theorem}\label{theo:main}
    Let $G$ be a finite group admitting a rigid, holomorphic and translation-free action on a 3-dimensional complex torus $T$ with finite fixed locus and such that the quotient $X=T/G$ has canonical singularities and $p_g=0$. Then $G$ is one of the following groups:
    \[
        \ZZ_9,\quad \ZZ_{14},\quad \ZZ_3^2=\langle h,k\rangle, \quad \ZZ_3^3=\langle h,k,g\rangle,\quad \ZZ_9\rtimes\ZZ_3=\langle g,h\mid h^3=g^3=1,\:hgh^{-1}=g^4\rangle.
    \]
    There are precisely 13 biholomorphism classes of quotients $T/G$. Table~\ref{table0} contains precisely one representative $Y_i$ of each class.
	
    All quotients admit rigid crepant terminalizations with numerically trivial canonical divisor and smooth rigid threefolds as resolutions.\\
    These 13 threefolds form 11 diffeomorphism classes, $Y_4\simeq_{\diff}Y_{4'}$ and $Y_{10}\simeq_{\diff}Y_{10'}$. Explicit diffeomorphisms are given by
    \begin{align*}
        Y_4\longrightarrow Y_{4'},\quad &(z_1,z_2,z_3)\mapsto (-z_1,\overline{z_3},\overline{z_2}),\\
        Y_{10}\longrightarrow Y_{10'},\quad &(z_1,z_2,z_3)\mapsto  (\overline{z_1},\overline{z_2},\overline{z_3})
    \end{align*}
    Furthermore, the homeomorphism and diffeomorphism classes coincide.\\
    All non simply connected quotients have $E^3/\langle \diag(\ze_3,\ze_3,\ze_3^2)\rangle$ as universal cover, which is not rigid but diffeomorphic to the rigid threefold $Z_2=E^3/\langle\ze_3\cdot\id\rangle$.
\end{Theorem}

\begin{center}
    \begin{table}[h!]
		{\footnotesize
		\bgroup\def\arraystretch{1.5}\begin{tabular}{|c|c|c|l|c|c|} \hline 
		$i$	&\emph{$G$} & \emph{ $\Lambda$} & \multicolumn{1}{c|}{action} & singularities & $\pi_1(Y_i)$ \\ \hline \hline
		1	&$\ZZ_9$ & $\Lam(\ze_9,\ze_9^4,\ze_9^7)$ & \begin{tabular}{l}$\Phi(1)(z)=\diag(\ze_9,\ze_9^4,\ze_9^7)\cdot z$\end{tabular}& \begin{tabular}{c}
		$8\times \tfrac{1}{3}(1,1,1)$\\ $3\times\tfrac{1}{9}(1,4,7)$ \end{tabular}& $\{1\}$ \\ \hline \hline
		2	&$\ZZ_{14}$ &  $\Lam(\ze_{14},\ze_{14}^9,\ze_{14}^{11})$& \begin{tabular}{l}$\Phi(1)(z)=\diag
		(\ze_{14},\ze_{14}^9,\ze_{14}^{11})\cdot z $\end{tabular} & \begin{tabular}{c} $9 \times \tfrac{1}{2}(1,1,1)$\\ $3\times\tfrac{1}{7}(1,2,4)$\\ $1\times \tfrac{1}{14}(1,9,11)$ \end{tabular}& $\{1\}$\\ \hline \hline
		3	&$\ZZ_3^2$ & $\ZZ[\ze_3]^3$ & \begin{tabular}{l}$\Phi(h)(z)=\diag(1,\ze_3^2,\ze_3^2)\cdot z+(t,t,t)$\\  $\Phi(k)(z)=\diag(\ze_3,\ze_3,\ze_3^2)\cdot z$\end{tabular}& $9\times \tfrac{1}{3}(1,1,2)$ & $\ZZ_3$\\ \hline
		4   & $\ZZ_3^2$ & $\ZZ[\ze_3]^3+\ZZ(t,t,0) $ & \begin{tabular}{l}$\Phi(h)(z)=\diag(1,\ze_3^2,\ze_3^2)\cdot z+\tfrac{1}{3}(1,1,3t)$\\  $\Phi(k)(z)=\diag(\ze_3,\ze_3,\ze_3^2)\cdot z$\end{tabular} &  $9\times \tfrac{1}{3}(1,1,2)$ & $\ZZ_3$\\ \hline
        4'   & $\ZZ_3^2$ & $\ZZ[\ze_3]^3+\ZZ(t,0,t) $ & \begin{tabular}{l}$\Phi(h)(z)=\diag(1,\ze_3^2,\ze_3^2)\cdot z+\tfrac{1}{3}(1,3t,2\ze_3^2)$\\  $\Phi(k)(z)=\diag(\ze_3,\ze_3,\ze_3^2)\cdot z$\end{tabular} &  $9\times \tfrac{1}{3}(1,1,2)$ & $\ZZ_3$\\ \hline
		5   & $\ZZ_3^2$ & $\ZZ[\ze_3]^3+\ZZ(t,t,t) $ & \begin{tabular}{l}$\Phi(h)(z)=\diag(1,\ze_3^2,\ze_3^2)\cdot z+\tfrac{1}{3}(1,1,2)$\\  $\Phi(k)(z)=\diag(\ze_3,\ze_3,\ze_3^2)\cdot z$\end{tabular} &  $9\times \tfrac{1}{3}(1,1,2)$ & $\ZZ_3$ \\ \hline
		6   & $\ZZ_3^2$ & $\ZZ[\ze_3]^3+\ZZ(t,t,t)+\ZZ(t,-t,0) $ & \begin{tabular}{l}$\Phi(h)(z)=\diag(1,\ze_3^2,\ze_3^2)\cdot z+\tfrac{1}{3}(1,1,2)$\\  $\Phi(k)(z)=\diag(\ze_3,\ze_3,\ze_3^2)\cdot z$\end{tabular} &  $9\times \tfrac{1}{3}(1,1,2)$ & $\ZZ_3$\\
		 \hline \hline
        7 & $\ZZ_3^2$ & $\ZZ[\ze_3]^3$ & \begin{tabular}{l}$\Phi(h)(z)=\diag(\ze_3,\ze_3,1)\cdot z+(t,t,t)$\\  $\Phi(k)(z)=\diag(\ze_3,\ze_3,\ze_3^2)\cdot z$\end{tabular} &\begin{tabular}{c} $9\times \tfrac{1}{3}(1,1,1)$\\ $9\times\tfrac{1}{3}(1,1,2)$\end{tabular} & $\{1\}$\\ 
        \hline
        8 & $\ZZ_3^2$ & $\ZZ[\ze_3]^3+\ZZ(t,t,t)$ & \begin{tabular}{l}$\Phi(h)(z)=\diag(\ze_3,\ze_3,1)\cdot z+\tfrac{1}{3}(1,1,1)$\\  $\Phi(k)(z)=\diag(\ze_3,\ze_3,\ze_3^2)\cdot z$\end{tabular} & \begin{tabular}{c}$9\times \tfrac{1}{3}(1,1,1)$\\ $9\times\tfrac{1}{3}(1,1,2)$\end{tabular} & $\{1\}$\\ 
        \hline \hline
        9 & $\ZZ_3^3$ & $\ZZ[\ze_3]^3$ & \begin{tabular}{l}$\Phi(h)(z)=\diag(1,\ze_3^2,\ze_3)\cdot z+(-t,-t,t)$\\  $\Phi(g)(z)=\diag(\ze_3,1,1)\cdot z +(-t,0,-t) $ \\$\Phi(k)(z)=\diag(\ze_3,\ze_3,\ze_3)\cdot z$\end{tabular} & \begin{tabular}{c}$3\times \tfrac{1}{3}(1,1,1)$\\ $9\times\tfrac{1}{3}(1,1,2)$\end{tabular}& $\{1\}$\\ 
        \hline
        10 & $\ZZ_3^3$ & $\ZZ[\ze_3]^3+\ZZ(t,t,0)$ & \begin{tabular}{l}$\Phi(h)(z)=\diag(1,\ze_3^2,\ze_3)\cdot z+\tfrac{1}{3}(-\ze_3^2,2,3t)$\\  $\Phi(g)(z)=\diag(\ze_3,1,1)\cdot z + \tfrac{1}{3}(-\ze_3^2,2\ze_3,0) $ \\$\Phi(k)(z)=\diag(\ze_3,\ze_3,\ze_3)\cdot z$\end{tabular} & \begin{tabular}{c}$3\times \tfrac{1}{3}(1,1,1)$\\ $9\times\tfrac{1}{3}(1,1,2)$\end{tabular} & $\{1\}$\\ 
        \hline
        10'& $\ZZ_3^3$ & $\ZZ[\ze_3]^3+\ZZ(t,t,0)$ & \begin{tabular}{l}$\Phi(h)(z)=\diag(1,\ze_3^2,\ze_3)\cdot z+\tfrac{1}{3}(\ze_3,\ze_3^2,3t)$\\  $\Phi(g)(z)=\diag(\ze_3,1,1)\cdot z + \tfrac{1}{3}(\ze_3^2,\ze_3^2,0) $ \\$\Phi(k)(z)=\diag(\ze_3,\ze_3,\ze_3)\cdot z$\end{tabular} & \begin{tabular}{c} $3\times \tfrac{1}{3}(1,1,1)$\\$9\times\tfrac{1}{3}(1,1,2)$ \end{tabular}& $\{1\}$\\ \hline\hline
        11 & $\ZZ_9\rtimes\ZZ_3$ & $\ZZ[\ze_3]^3$ & \begin{tabular}{l} $\Phi(h)(z)=\diag(\ze_3,1,\ze_3^2)\cdot z + (t,t,t)$\\ $\Phi(g)(z)=\begin{pmatrix}
            0 & 1 & 0\\ 0 & 0 & 1\\ \ze_3 & 0 & 0
        \end{pmatrix}\cdot z$\end{tabular} &\begin{tabular}{c} $2\times \tfrac{1}{3}(1,1,1)$\\$3\times\tfrac{1}{9}(1,4,7)$ \end{tabular}& $\{1\}$ \\ \hline
        \end{tabular}\egroup}
        \caption{Quotients with $p_g=0$. In the table, $t:=(1+2\ze_3)/3$ and $\Lam(\ze_9,\ze_9^4,\ze_9^7)$  has basis $\{(\ze_9^k,\ze_9^{4k},\ze_9^{7k})\mid \gcd(k,9)=1\}$ and
         $\Lam(\ze_{14},\ze_{14}^9,\ze_{14}^{11})$ has  basis $\{(\ze_{14}^k,\ze_{14}^{9k},\ze_{14}^{11k})\mid \gcd(k,14)=1\}$.}
         \label{table0}
    \end{table}
	\end{center}	

Clearly, biholomorphism classes with different geometric genus can not coincide, but diffeomorphisms might exists. All possible relations and the final classification result is summarized as follows:

\begin{Theorem}
    Let $G$ be a finite group admitting a rigid, holomorphic and translation-free action on a 3-dimensional complex torus $T$ with finite fixed locus and such that the quotient $X=T/G$ has canonical singularities. Then $G$ is one of the following groups:
    \[
        \ZZ_3,\quad\ZZ_7,\quad\ZZ_9,\quad\ZZ_{14},\quad\ZZ_3^2,\quad\ZZ_3^3,\quad \He(3),\quad \ZZ_9\rtimes\ZZ_3.
    \]
    The quotients $X=T/G$ form 21 biholomorphism classes, which can be represented by $Z_1,\ldots,Z_8$, $Y_1,\ldots,Y_{10'}$ from the Tables~\ref{tab:CalabiYau} and \ref{table0}, and 15 diffeomorphism classes
    \begin{align*}
        &Z_1,\quad Z_2,\quad Z_7,\quad Z_8, \quad Y_1, \quad Y_2, \quad Y_7, \quad Y_8,\quad Y_9,\quad Y_{11}\\
        &Z_3\simeq_{\diff}Y_3,\quad Z_4\simeq_{\diff}Y_4\simeq_{\diff}Y_{4'},\quad Z_5\simeq_{\diff}Y_5,\quad Z_6\simeq_{\diff}Y_6,\quad Y_{10}\simeq_{\diff}Y_{10'}
    \end{align*}
    Explicit diffeomorphisms $Z_k\to Y_k$ for $k=3,\ldots,6$ are given by $(z_1,z_2,z_3)\mapsto (z_1,z_2,-\overline{z_3})$.\\
    The homeomorphism and diffeomorphism classes coincide.
\end{Theorem}

We want to mention that partial classification results were already obtained by the first author and Bauer in \cite{BG21}. Here, the authors made the assumption that the torus is a product of three elliptic curves and the action of the group on the product is diagonal and faithful on each factor. This allowed them to use product quotients techniques. They found precisely the examples $Y_3,Y_5,Y_7$ and $Y_8$ in our main Theorem \ref{theo:main} and the Calabi-Yau threefolds $Z_3$ and $Z_5$ in Table~\ref{tab:CalabiYau}.\\
We want to point out that all tori occurring in our classification are (abstractly) isomorphic to a pro\-duct of three elliptic curves: either three copies of Fermat's elliptic curve $E:=\CC/\ZZ[\ze_3]$ or three copies of $E_{u_7}:=\CC/\ZZ[\ze_7+\ze_7^2+\ze_7^4]$, but conjugating the actions with these isomorphisms lead to non-diagonal actions, which do not fit in the setup of \cite{BG21}.\\

Our method for the classification of the quotients relies heavily on the fact that the fundamental group of the regular locus of such a quotient is a crystallographic group. This allows us to use Bieberbach's structure theorems \cite{Bieber1}, \cite{Bieber2} and group cohomology, which are standard tools in the classification of compact flat manifolds and free torus quotients (cf.~\cite{charlap}, \cite{HalendaLutowski}, \cite{DG}) and were extended to the singular case in \cite{GK}.

\bigskip
The paper is organized as follows:
 in Section~\ref{sec:preliminaries}, we collect some preliminaries concerning torus quotients and introduce the notion of orbifold fundamental groups, which are crystallographic in this setup. Furthermore, we recall that the rigidity of a torus quotient is encoded in  the linear part of the action and explain why rigid torus quotients have vanishing irregularities $q_1$ and $q_2$. 
The list of possible Galois groups is determined in Section~\ref{sec:classgroup}. For this purpose, we first analyze and bound  the order of the stabilizers which turn out to be always cyclic (cf. Theorem~\ref{theo:groups}), which enables us to use Morrison's classification of isolated canonical cyclic quotient singularities \cite{Morrison}. 
Knowing the possible types of singularities,  the orbifold Riemann-Roch formula and methods from group and representation theory allow us to deduce the groups and the linear parts of the actions. Section \ref{sec:classQuot} is devoted to the fine classification of the quotients up to biholomorphism and diffeomeorphism applying  techniques as mentioned above.
In Section~\ref{sec:FundGroups}, we study the structure of the fundamental groups of torus quotients and their universal covers and compute them explicitly for our quotients. 
 In Section~\ref{sec:resolution}, we use methods from toric geometry locally to prove that our quotients admit rigid crepant terminalizations and rigid resolutions of their singularities.
 The proof of our main theorem~\ref{theo:main} is provided in 
Section~\ref{sec:proofMain}. Here, we finally put the results of the previous sections together.\\

During the classification of the groups as well as for the quotients, some computations were performed using the computer algebra system MAGMA (\cite{MAGMA}). The code can be found on the website:

\begin{center}
  \url{https://www.komplexe-analysis.uni-bayreuth.de/de/team/gleissner/index.html}  
\end{center}


\section{Preliminaries}\label{sec:preliminaries}

Let $T = \CC^n/\Lam$ be a complex torus and $G$ a finite group acting faithfully and holomorphically on $T$ via 
\[\Phi\colon G\longhookrightarrow \Bihol(T).\]
Since holomorphic maps between complex tori are affine, we can decompose the action into its \textit{linear part} $\rho$ and its \textit{translation part} $\tau$, i.e., $\Phi(g)(z) = \rho(g)z + \tau(g)$ for all $g\in G$. The homomorphism
\begin{align*}
	\rho \colon G \longrightarrow \GL(\CC^n),\quad g\mapsto \rho(g),
\end{align*}
is called the \textit{analytic representation}. Since the quotient of a complex torus by a finite group of translations is again a complex torus, we can and will always assume that $G$ acts without translations, equivalently $\rho$ is faithful. In contrast to the linear part, the translation part $\tau\colon G\to T$ is not a homomorphism, but a 1-cocycle,
thus, it defines an element of  the first group cohomology
\[H^1(G,T)=\frac{\{\tau\mid \tau(gh)=\rho(g)\tau(h)+\tau(g)\}}{\{\tau\mid \exists\, d\in T\colon \tau(g)=\rho(g)d-d\}}. \]
Here,  we view $T$ as a $G$-module via the action of $\rho$. Up to conjugation by a translation, the action $\Phi$ is uniquely determined by $\rho$ and the cohomology class of $\tau$. Conversely, the choice of a cohomo\-lo\-gy class together with $\rho$ yields an action on $T$, which is well-defined up to conjugation by a translation.\\

The group of lifts
\[
    \Gamma=\pi_1^{orb}(T,G):=\{\gamma\colon \CC^n\to\CC^n\mid \exists \:g\in G \colon g\circ p =   p \circ \gamma\},
\]
where $p\colon \CC^n\to T$ is the quotient map, is called the \emph{orbifold fundamental group}. 
 If the action of $G$ on $T$ is free in codimension at least 1, then the orbifold fundamental group $\Gamma$ coincides with the fundamental group of the regular locus of the quotient $X=T/G$.
Since $G$ is finite, we can assume without loss of generality that the analytic representation $\rho$ is unitary. Hence, its decomplexification is orthogonal and $\Gamma$ can be considered as a cocompact and discrete subgroup of  the Euclidean group $\mathbb{E}(2n):=\RR^{2n}\rtimes \Ort(2n)$.

\begin{definition}
    A discrete cocompact subgroup of $\mathbb{E}(n)$ is called a \emph{crystallographic group}.
\end{definition}

Bieberbach's theorems describe the structure of crystallographic groups:

\begin{Theorem}[\cite{Bieber1}, \cite{Bieber2}]
    The  translation subgroup $\Lambda:=\Gamma \cap \mathbb R^n$ of a crystallographic group  $\Gamma \leq \mathbb{E}(n)$ is a 
	lattice  of rank $n$ and the quotient $\Gamma/\Lambda$ is finite.  All normal abelian subgroups of $\Gamma$ are contained in $\Lambda$. Furthermore, an isomorphism between two crystallographic groups is given by conjugation with an affine transformation.
\end{Theorem}

Since $G$ acts without translations on $T$, the lattice $\Lambda$ of the torus is precisely the subgroup of translations of $\Gamma$.  Thus, the sequence
\[
    0\longrightarrow \Lambda\longrightarrow \Gamma \longrightarrow G\longrightarrow 1
\]
is exact.

Every biholomorphism (or even homeomorphism) $f\colon X\to X'$ between two torus quotients obtained by actions free in codimension 1 induces a biholomorphism (or homeomorphism) between the regular loci of the quotients and therefore an isomorphism between the orbifold groups. Based on this observation, the following proposition can be derived:

\begin{prop}[cf. \cite{GK}*{Proposition~3.6}]\label{prop:ConsBieb}
	Let $\Phi\colon G\to \Bihol(T)$ and $\Phi'\colon G'\to\Bihol(T')$ be translation-free holomorphic actions of finite groups $G$ and $G'$ on $n$-dimensional complex tori $T$ and $T'$ where $n\geq 2$. If the actions are free in codimension 1 and the quotients $X=T/G$ and $X'=T'/G'$ are homeomorphic, then:
	\begin{enumerate}
		\item The groups $G$ and $G'$ are isomorphic.
		\item There exists an affine transformation $\alpha\in \AGL(2n,\RR)$ inducing diffeomorphisms $\widehat{\alpha}$ and $\widetilde{\alpha}$, such that the following diagram commutes:
	 \[
	 \begin{tikzcd}
	 	T \arrow{d}\arrow{r}{\widetilde{\alpha}} & T' \arrow{d}\\
	 	X\arrow{r}{\widehat{\alpha}} & X'.
	 \end{tikzcd}
	\]
 	\end{enumerate}
	Furthermore, any biholomorphism $f\colon X\to X'$ lifts to a biholomorphism of the tori, hence it is induced by an affine transformation $\alpha\in\AGL(n,\CC)$.
\end{prop}

Let $f\colon X\to X'$ be a homeomorphism induced by an affine transformation $\alpha(x)=Cx+d$. Then the commutativity of the diagram in Proposition~\ref{prop:ConsBieb} is equivalent to the existence of an isomorphism $\varphi\colon G\to G'$ such that
\[ (\mathrm{a})\: C\rho_\RR(g)C^{-1}=\rho'_\RR(\varphi(g)) \qquad \mathrm{and}\qquad (\mathrm{b})\: (\rho'_\RR(\varphi(g))-\id)d=C\tau(g)-\tau'(\varphi(g)) \]
hold for all $g\in G$, where the second item is an equation holding on $T'$. Note that $\varphi=:\varphi_C$ is uniquely determined by $C$.\\
If we consider $T$ and $T'$ as $G$ and $G'$-modules via $\rho_\RR$ and $\rho'_\RR$, then by item (a), the matrix $C$ induces  a twisted equivariant module isomorphism $C\colon T\to T'$. Item (b) tells us that the cocyles $\tau'$ and $$C\ast\tau:=C\cdot(\tau\circ\varphi_C^{-1})$$ differ by a coboundary.

\begin{notation}
	Let $T=\CC^n/\Lam, T'=\CC^n/\Lam'$ and $\Phi, \Phi'$ be holomorphic actions of a finite group $G$ on $T$ and $T'$, having linear parts $\rho,\rho'\colon G\to\GL(n,\CC)$ and translation parts $\tau$ and $\tau'$, respectively. Denote by $\rho_\RR$ and $\rho'_\RR$ the decomplexifications of $\rho$ and $\rho'$, respectively.\\
	We define
	\[
		\sN_\RR(\Lam,\Lam'):=\{C\in\GL(2n,\RR)\mid C \Lam=\Lam',\: C\cdot\im(\rho_\RR)=\im(\rho'_\RR)\cdot C\}
	\]
	and
	\[
		\sN_\CC(\Lam,\Lam'):=\sN_\RR(\Lam,\Lam')\cap \GL(n,\CC).
	\]
\end{notation}

In summary, we have:

\begin{prop}\label{prop:classes}\
    \begin{enumerate}
        \item Two quotients $X=T/G$ and $X'=T'/G$ are homeomorphic (biholomorphic) if and only if there exists a matrix $C\in \sN_\RR(\Lam,\Lam')$ ($C\in\sN_\CC(\Lam,\Lam'))$ such that $C\ast \tau$ and $\tau'$ belong to the same cohomology class in $H^1(G,T')$.
        \item If $X=T/G$ and $X'=T'/G$ are homeomorphic, then $T$ and $T'$ are isomorphic as $G$-modules up to an automorphism  of $G$.
    \end{enumerate}
\end{prop}

\begin{rem}
    The cocycles $C\ast \tau$ and $\tau'$ belong to the same cohomology class in $H^1(G,T')$ if and only if there there exists an element $d\in T'$ such that for all $g\in G$, it holds $(C\ast\tau-\tau')(g)=\rho'(g)d-d$. Conjugation by the affinity $\alpha(x)=Cx+d$ induces isomorphisms:
    \[\begin{tikzcd}
	 	0 \arrow{r} & \Lambda \arrow{d}{C} \arrow{r} & \pi_1^{\orb}(T,G) \arrow{d}{\con_{\alpha}}\arrow{r} & G \arrow{r}\arrow{d}{\varphi_C} & 1\\
	 	0 \arrow{r} & \Lambda'\arrow{r} & \pi_1^{\orb}(T',G)\arrow{r} & G \arrow{r} & 1
	 \end{tikzcd}\]
    Conversely, every isomorphism of the orbifold fundamental groups is given by conjugation with an affinity yielding $C$ and $d$ as above.
\end{rem}

In the special case where $\rho=\rho'$  and  $T=T'$, the sets $\mathcal N_{\mathbb R}(\Lambda,\Lambda)$ and $\mathcal N_{\mathbb C}(\Lambda,\Lambda)$ are the normalizers of $\im(\rho_{\mathbb R})$ in  the group of linear diffeomorphism or biholomorphisms of $T$. For simplicity, we denote them by $\mathcal N_{\mathbb R}(\Lambda)$ and $\mathcal N_{\mathbb C}(\Lambda)$. 
They  act on $H^1(G,T)$ by $C \ast \tau$. The quotients corresponding to $\tau$ and $\tau'$ are homeomorphic (or biholomorphic) if and only if they belong to the same orbit under this action.\\

Next, we collect some notions and tools from deformation theory, which are necessary to study \emph{rigid} torus quotients, the main objects of this article.

\begin{definition}\label{rigidity}
	Let $X$ be a compact complex space.
	\begin{enumerate}
		\item A \textit{deformation} of $X$ consists of the following data:
		\begin{itemize}
			\item a flat and proper holomorphic map  $\pi \colon \mathfrak X \to B$ of connected complex spaces,
			\item a point $0 \in B$,
			\item an isomorphism $\pi^{-1}(\{0\}) \simeq X$. 
		\end{itemize}
		 \item We call $X$ \textit{(locally) rigid} if for every deformation  $\pi \colon \mathfrak X \to B$ of $X$, there is an open neighborhood $U \subset B$ of $0$ such that $X\simeq \pi^{-1}(t)$ for all $t\in U$.
		\item We call $X$ \textit{infinitesimally rigid} if $\Ext^1(\Omega^1_X,\sO_X) = 0$.
	\end{enumerate}
\end{definition}

\begin{rem}
	If $X$ has dimension at least 3 and only isolated quotient singularities, then the sheaf $\SheafExt^1(\Omega^1_X,\sO_X)$ is trivial due to a result of Schlessinger \cite{schlessinger}. The short-term exact sequence of the local-to-global $\Ext$ spectral sequence gives us therefore an isomorphism
	\[H^1(X,\Theta_X)\simeq \Ext^1(\Omega_X^1,\sO_X),\]
	where $\Theta_X=\SheafHom(\Omega_X^1,\sO_X)$ denotes the holomorphic tangent sheaf. In particular, $X$ is infinitesimally rigid if and only if $H^1(X,\Theta_X)$ is trivial in analogy to compact complex manifolds.\\
	If $X=Y/G$ is a quotient of complex manifold $Y$ by an action of a finite group $G$ that is free in codimension 1, then $H^1(X,\Theta_X)=H^1(Y,\Theta_Y)^G$.
\end{rem}

\begin{definition}
	A holomorphic action of a finite group $G$ on a complex manifold $Y$ is called \textit{in\-fini\-tesi\-mally rigid} if
	\[H^1(Y,\Theta_Y)^G=0.\]
\end{definition}

It is known that every infinitesimally rigid compact complex space is (locally) rigid. The converse does not hold in general, even if we restrict to manifolds. However, in the situation of torus quotients, these two notions coincide:

\begin{prop}[cf. \cite{DG}*{Proposition~2.5, Corollary~2.6}]\label{prop:Rigid}
	Let $X=T/G$ be a torus quotient of dimension at least $3$ by an action with at most isolated fixed points. Then:
 \begin{enumerate}
     \item $X$ is rigid if and only if it is infinitesimally rigid.
     \item $X$ is infinitesimally rigid if and only if the analytic representation $\rho$ and its complex conjugate $\overline{\rho}$ don't have any subrepresentations in common.
 \end{enumerate}
\end{prop}

\begin{rem}
	Due to \cite{CD20}, any complex torus quotient has an algebraic approximation. In particular, rigid torus quotients are projective.
\end{rem}

\begin{rem}
	Let $f\colon \hat{X}\to X$ be a resolution of $X=T/G$. Since quotient singularities are rational, the irregularities
	\[q_i(\hat{X})=h^i(\hat{X},\sO_{\hat{X}})\qquad\mathrm{and}\qquad q_i(X)=h^i(X,\sO_X)\] 
	coincide. Let $\chi$ be the character of the analytic representation $\rho$. 
 As $H^0(\hat{X},\Omega^i_{\hat{X}})\simeq H^0(T,\Omega_T)^G$, we can compute the irregularities $q_i$ as follows:
	\[q_i(X)=q_i(\hat{X})=\dim_\CC(H^0(T,\Omega^i_T)^G)=\langle \wedge^i(\overline{\chi}),\chi_{triv}\rangle.\]
Using the formula  $\chi^2=\wedge^2(\chi)+\Sym^2(\chi)$ and $\langle \chi, \overline{\chi} \rangle =0$ from the rigidity of the action, we conclude that  $q_1=q_2=0$.

\end{rem}


\section{Classification of the Groups}\label{sec:classgroup}

This section is devoted to the classification of all finite groups  $G$ acting holomorphically with isolated fixed points on a complex torus $T$ of dimension $3$ such that $X=T/G$ is rigid with canonical singularities. We assume that the action is translation free, i.e., it has a faithful linear part
\[
    \rho\colon G\longhookrightarrow\GL(3,\CC).
\]
During the classification process, we will frequently make use of two basic observations:
\begin{Remark} Given an action of a finite group $G$ as above, then:
    \begin{enumerate}
        \item The restriction to every subgroup $U$ shares the same properties apart from the rigidity of the quotient $T/U$.
        \item For all $g\in G$, it holds:
        \begin{itemize}
            \item If $g$ acts freely, then $1$ is an eigenvalue of $\rho(g)$.
            \item If $g$  has fixed points and order $d$, then all the eigenvalues of $\rho(g)$ must be primitive $d$-th roots of unity since otherwise, the fixed locus of some power of $g$ has positive dimension.
        \end{itemize}     
    \end{enumerate}
\end{Remark}

From now on, we fix a finite group $G$ and assume that it admits an action with the above properties. First, we determine the possible orders of the elements in such a group.

\begin{lemma}\label{le:orders}
	Let $g\in G$ be a non-trivial element.
	\begin{enumerate}
		\item If $g$ acts freely on $T$, then $\ord(g)\in \{ 2,\ 3,\ 4,\ 5,\ 6,\ 8,\ 10,\ 12\}$.
		\item If $g$ acts with fixed points, then $\ord(g)\in\{2,\ 3,\ 4,\ 6,\ 7,\ 9,\ 14,\ 18\}$.
	\end{enumerate}
	In particular, $\ord(g)\in\{2,\ldots,10,12,14,18\}$, and elements of order $7,9,14,18$ always have fixed points and elements of order $5,8,10,12$ always act freely.
\end{lemma}

\begin{proof}
    Let $d:=\ord(g)$. Assume first that $g$ acts freely. Then $\rho(g)$ has eigenvalue 1. If the other two eigenvalues have the same order, then $\varphi(d)\leq 4$, where $\varphi$ denotes the Euler totient function. Otherwise, the orders $d_1, d_2$ of the other two eigenvalues fulfill $\varphi(d_1)+\varphi(d_2)\leq 4$ (cf. \cite{BGL}*{Proposition~3.1} or \cite{Dem}*{Lemma~3.1.6}) and $d=\lcm(d_1,d_2)$. It is now easy to determine all possible values for $d$.\\
    If $g$ acts with fixed point, then all its eigenvalues are primitive $d$-th roots of unity. In this case, $\varphi(d)$ divides $6$ by the same proposition. This implies $d\in\{2,3,4,6,7,9,14,18\}$.
\end{proof}

\begin{lemma}\label{le:PGroups}
	Assume that $G$ contains an abelian subgroup $U$ such that every element in $U$ acts non-freely. Then $U$ is cyclic.
\end{lemma}

\begin{proof}
	Since $U$ is abelian, we can assume that $\rho$ restricted to $U$ is the direct sum of three 1-dimensional representations. Each of them must be faithful because the identity is the only element having eigenvalue one. Hence, $U$ is cyclic.
\end{proof}

\begin{cor}\label{cor:7groupsCyclic}
	If $G$ has a $7$-Sylow subgroup $S_7$, then $S_7$ is cyclic of order $7$.
\end{cor}

\begin{proof}
	By Sylow's theorem, it suffices to exclude that $G$ has a subgroup $U$ of order $7^2$. Lemma~\ref{le:orders} ensures that $U$ is not cyclic and every element acts with fixed points. Thus, $U\simeq\ZZ_7^2$, which contradicts Lemma~\ref{le:PGroups}.
\end{proof}

Now, we are ready to determine the possible non-trivial stabilizer groups and the corresponding singularities of the quotient. It turns out that all of them are cyclic. 

\begin{theo}\label{theo:Singularities}
	For all $p\in T$, the stabilizer group $\Stab(p)$ is cyclic of order $1,2,3,4,6,7,9$ or $14$.\\
	In particular, the quotient $X$ has only isolated cyclic quotient singularities. The possible types are $\tfrac{1}{d}(1,1,d-1)$, where $d=2,3,4,6$, and $\tfrac{1}{3}(1,1,1)$, $\tfrac{1}{7}(1,2,4)$, $\tfrac{1}{9}(1,4,7)$ and $\tfrac{1}{14}(1,9,11)$.	
\end{theo}

\begin{proof}
	Let $p\in T$ be a point with non-trivial stabilizer $H:=\Stab(p)$. Moving the origin of $T$, we may assume that  $H$ acts linearly. 
    In particular, every element of $H$ has $0\in T$ as fixed point. First, we prove that $H$ is cyclic. By Lemma~\ref{le:PGroups}, it is enough to show that $H$ is abelian.\\ 
    We start with summing some relevant properties of $H$ and its elements:
	\begin{enumerate}[(1)]
		\item For every non-trivial element $g\in H$, the matrix $\rho(g)$ has $1$ not as eigenvalue since the fixed points are assumed to be isolated.
		\item By Lemma~\ref{le:orders}, for all $g\in G$, it holds $\ord(g)\in\{1,\ 2,\ 3,\ 4,\ 6,\ 7,\ 9,\ 14, \ 18 \}$. In particular, $\lvert H\rvert = 2^a\cdot 3^b\cdot 7^c$.
		\item Let $g\in H$ be an element of order $2$. Then $\rho(g)=-\id$.
            In particular, $H$ contains at most one element of order $2$ since $\rho$ is faithful.
		\item By Corollary~\ref{cor:7groupsCyclic}, we have $c\in \{0,1\}$.
	\end{enumerate}
	Next, we analyze the $2$- and $3$-Sylow subgroups of $H$. We claim that they are cyclic of order $2$ or $4$, and $3$ or $9$, respectively (if existend). By Sylow's theorem and Lemma~\ref{le:PGroups}, it is enough to show that $H$ has no subgroups of order $p^3$ for $p=2,3$. Such a subgroup $U$ cannot be cyclic by item (2) and hence not abelian by Lemma~\ref{le:PGroups}. If $p=2$, then the three-dimensional representation $\rho$ restricted to this subgroup has a 1-dimensional subrepresentation, which has to be faithful by item~(1) -- a contradiction. If $p=3$, then $U$ is either $\He(3)$ or $\ZZ_9\rtimes \ZZ_3$. Both of them contain $\ZZ_3^2$ as a subgroup contradicting Lemma~\ref{le:PGroups}. In particular, $a,b\in\{0,1,2\}$.\\
	  Finally, we show that there is no non-abelian group fulfilling all these conditions. Note that if $H$ is not abelian,  the representation $\rho$ need to be irreducible. In particular, $3=\chi_\rho(1)$ has to divide the group order, so $b\neq 0$.
	\begin{itemize}[itemsep=6pt]
		\item \underline{$c=0$}: If $a=0$, then $H$ is  abelian.\\
        If $a=1$, then the $3$-Sylow subgroup is normal due to Sylow's theorems. By item (3), $H$ has only one $2$-Sylow subgroup $\ZZ_2$, which is normal (its generator acts with $-\id$), too.  Hence, $H$ is  abelian.\\
		If $a=2$, then the only groups admitting an irreducible representation of dimension $3$ are $\sA_4$, $\sA_4\times\ZZ_3$ and $\ZZ_2^2\rtimes\ZZ_9$, all of which have more than one element of order $2$.
		\item \underline{$c=1$}: The only groups having at most one element of order $2$ and admitting an irreducible representation of dimension $3$ contain $\ZZ_7\rtimes\ZZ_3$  or $\ZZ_7\rtimes \ZZ_9$ as a subgroup. Thus, we only have to exclude these groups. Up to complex conjugation and equivalence of representations, the only irreducible $3$-dimensional representation of $\ZZ_7\rtimes\ZZ_3=\langle t,s\mid t^7=s^3=1,\:sts^{-1}=t^4\rangle$ is given by
		\[
		s\mapsto \begin{pmatrix}
			0 & 1 & 0 \\ 0 & 0 &1 \\ 1 & 0 &0
		\end{pmatrix},\qquad t\mapsto \begin{pmatrix}
			\ze_7^4 & 0 & 0\\ 0&\ze_7^2 & 0 \\ 0 & 0 & \ze_7
		\end{pmatrix}.
		\]
		But then the matrix of $s$ has eigenvalue $1$. The group $\ZZ_7\rtimes\ZZ_9$ has an element of order $21$.
	\end{itemize}
    Thus, $H$ is cyclic and by Lemma~\ref{le:orders}, its order $m$ belongs to $\{2,3,4,6,7,9,14,18\}$.  By Morrison's classification (cf. \cite{Morrison}), each isolated cyclic canonical quotient singularity is isomorphic to precisely one of the following:
    \begin{itemize}[itemsep=6pt]
        \item $\tfrac{1}{m}(1,a,m-a)$, where $\gcd(m,a)=1$ (terminal)
        \item $\tfrac{1}{m}(1,a,m-a-1)$, where $\gcd(m,a)=\gcd(m,a+1)=1$ (Gorenstein)
        \item $\tfrac{1}{9}(1,4,7)$ or $\tfrac{1}{14}(1,9,11)$.
    \end{itemize}
    The condition $\gcd(m,a)=\gcd(m,a+1)=1$ in the Gorenstein case implies that $m$ is odd. 
    Note that for a linear automorphism $\alpha\in\Aut(T)$ of order $m$ with only primitive $m$-th roots of unity as eigenvalues, the function
    \[\mu_m^\ast\longrightarrow\ZZ,\quad\zeta \mapsto \mult(\zeta)+\mult(\overline{\zeta}),\]
	is constant, where $\mult(\zeta)$ denotes the multiplicity of $\zeta$ as eigenvalue of $\alpha$ and $\mu_m^\ast$ denotes the set of primitive $m$-th roots of unity (cf. \cite{Dem}*{Lemma~3.1.6}).\\
    In the terminal case, each generator of the stabilizer has two eigenvalues that are complex conjugate to each other. Thus, $\varphi(m)\leq 2$ or equivalently $m\in\{2,3,4,6\}$.\\
    Analyzing the remaining cases yields the singularities in the theorem. Note that there is no singularity of order $18$ since $18$ is even and $\varphi(18)=6$.
 \end{proof}

The main result of this section is:

\begin{theo}\label{theo:groups}
    Let $G$ be a finite group acting holomorphically, without translations and isolated fixed points on a complex torus $T$ of dimension $3$ such that $X=T/G$ is rigid with canonical singularities.
	\begin{enumerate}
		\item If $p_g(X)=1$, then $G\simeq \ZZ_7,\:\ZZ_3,\:\ZZ_3^2$ or $\He(3)$.
		\item If $p_g(X)=0$, then $G\simeq \ZZ_9,\:\ZZ_{14},\:\ZZ_3^2,\:\ZZ_3^3$ or $\ZZ_9\rtimes\ZZ_3$.
	\end{enumerate}
\end{theo}

\begin{Remark}
    We point out that the classification of the groups and the analytic representations in the case $p_g(X)=1$ was already achieved by Oguiso and Sakurai (\cite{OguisoQuotientType}*{Theorem~3.4}). Instead of the rigidity, they assumed  that the action has non-empty (isolated) fixed locus and that the quotients have vanishing irregularities $q_1$ and $q_2$. From their description of the analytic representations, the rigidity follows immediately from Proposition~\ref{prop:Rigid}. Conversely, any rigid action on a 3-dimensional torus has fixed points by \cite{DG}*{Theorem~1.1}.  In the rest of the section, we therefore only need to consider the case $p_g(X)=0$.
\end{Remark}

A first step towards the classification of the groups in the case $p_g(X)=0$ is to analyze the possible baskets of singularities. For this purpose, we make use of a relative version of the orbifold Riemann-Roch formula (cf. \cite{Reid87}), adapted to our setup (cf. \cite{Gle16}*{Section~4.3} for a similar situation). 

\begin{Proposition}\label{prop:CondNumberSing}
	If $p_g(X)=0$, then
	\[
		1=\tfrac{1}{16}N_2+\tfrac{1}{9}N_3+\tfrac{5}{32}N_4+\tfrac{35}{144}N_6+\tfrac{1}{3}N_9+\tfrac{7}{16}N_{14},
	\]
    where $N_d$ denotes the number of singularities of type $\tfrac{1}{d}(1,1,d-1)$ for $d=2,3,4,6$ and $N_9$ and $N_{14}$ the number of singularities of type $\tfrac{1}{9}(1,4,7)$ and $\tfrac{1}{14}(1,9,11)$, respectively.
\end{Proposition}

\begin{proof}
    The orbifold Riemann-Roch formula (cf. \cite{Reid87}) reads:
    \[
        \chi(\mathcal{O}_X)=\frac{1}{24}\left(-K_X\cdot c_2(X)+\sum_{x\, ter}\frac{m_x^2-1}{m_x}\right),
    \]
    where the sum runs over all terminal singularities $\tfrac{1}{m_x}(1,a_x,m_x-a_x)$ of a crepant terminalization of $X$, which we obtain by looking locally at each isolated singular point. The Gorenstein singularities have a crepant resolution, so they don't contribute. 
    The crepant terminalization of $\tfrac{1}{9}(1,4,7)$ consists of  three copies of $\tfrac{1}{3}(1,1,2)$ and the one of $\tfrac{1}{14}(1,9,11)$ of seven nodes $\tfrac{1}{2}(1,1,1)$ (cf. Section~\ref{sec:resolution}). Since the remaining singularities are all terminal (cf. Theorem~\ref{theo:Singularities}), we don't need to modify them.
    
    By the rigidity of the action, we have $q_1(X)=q_2(X)=0$, thus, $\chi(\sO_X)=1$. Moreover, the intersection product $K_X\cdot c_2(X)$ is 0  since $\lvert G\rvert\cdot K_X\sim_{lin} 0$.  Hence, the claim follows.
\end{proof}

\begin{Corollary}\label{cor:BasketsSing}
    The candidates for the values of $[N_2,N_3,N_4,N_6,N_9,N_{14}]$ are
    \begin{center}
        \begin{minipage}{0.31\textwidth}
            \begin{tabular}{c|c}
                 $k$ & $[N_2,N_3,N_4,N_6,N_9,N_{14}]$ \\ \hline
                 $1$ & $[0,1,2,1,1,0]$\\
                 $2$ & $[0,4,2,1,0,0]$\\
                 $3$ & $[1,0,6,0,0,0]$ \\
                 $4$ & $[2,0,0,0,0,2]$ \\
                 $5$ & $[4,0,2,0,0,1] $
            \end{tabular}
        \end{minipage}
        \begin{minipage}{0.31\textwidth}
            \begin{tabular}{c|c}
                 $k$ & $[N_2,N_3,N_4,N_6,N_9,N_{14}]$ \\ \hline
                 $6$ & $[5,1,0,1,1,0]$\\
                 $7$ & $[5,4,0,1,0,0] $\\
                 $8$ & $[6,0,4,0,0,0]$ \\
                 $9$ & $[9,0,0,0,0,1]$ \\
                 $10$ & $[11,0,2,0,0,0]$
            \end{tabular}
        \end{minipage}
        \begin{minipage}{0.31\textwidth}
            \begin{tabular}{c|c}
                 $k$ & $[N_2,N_3,N_4,N_6,N_9,N_{14}]$ \\ \hline
                 $11$ & $[16,0,0,0,0,0]$\\
                 $12$ & $ [0,0,0,0,3,0]$\\
                 $13$ & $[0,3,0,0,2,0]$\\
                 $14$ & $[0,6,0,0,1,0]$\\
                 $15$ & $[0,9,0,0,0,0]$
            \end{tabular}
        \end{minipage}
    \end{center}
\end{Corollary}

Next, we want to derive a formula which allows us to compute the order of the group $G$ in terms of the $N_i$.
If the image of the analytic representation $\rho$ contains non-trivial scalar matrices, we can derive such a formula from the Lefschetz fixed-point formula, which has a particularly simple shape on complex tori:

\begin{lemma}[\cite{BL}, Corollary~13.2.4, Proposition~13.2.5]\label{le:Fix}
	 Let $T$ be a complex torus of dimension $n$ and $\alpha\in\Aut(T)$ an automorphism of order $d$ such that $\langle \alpha\rangle$ acts with isolated fixed points. Then:
    \[
        \#\Fix(\alpha)=
            \begin{cases}
                p^{2n/\varphi(d)}, & \makebox{if}\: d=p^s \, \makebox{for some prime}\: p,\\
                1, & \makebox{else}
            \end{cases}.
    \]
\end{lemma}

\begin{Lemma}\label{le:Lefschetz}\
    \begin{enumerate}
        \item If $-\id\in\im(\rho)$, then $$2^6 = \lvert G\rvert \cdot (\tfrac{1}{2} N_2 + \tfrac{1}{4} N_4 + \tfrac{1}{6} N_6 + \tfrac{1}{14} N_{14}).$$
        \item If $\ze_3\cdot\id\in\im(\rho)$, then $$3^3=\lvert G\rvert\cdot (\tfrac{1}{3}N_{3,gor}+\tfrac{1}{9}N_9),$$ where $N_{3,gor}$ denotes the number of singularities of type $\tfrac{1}{3}(1,1,1)$.
    \end{enumerate}
\end{Lemma}

\begin{proof}
    We only give a proof for the first statement. The reasoning for the second is similar. Let $g\in G$ be the unique element with $\rho(g)=-\id$. The fixed points of $g$ are precisely the elements in $T$ with stabilizer of even order. Thus:
	\[
		\Fix(g)=\bigsqcup_{j\in\{2,4,6,14\}}\{y\in T \mid \Stab(y) \simeq \ZZ_j\}.
	\]
	By Lemma~\ref{le:Fix}, $g$ has $2^6$ fixed points. If $x$ is a singularity of order $j$, then the fiber of $x$ under the projection map $\pi\colon T\to X$ contains $\lvert G\rvert/j$ elements, all having a stabilizer group isomorphic to $\ZZ_j$.
\end{proof}

\begin{Remark}\label{rem:centralElements}
    If there exists an even $j$ such that $N_j\neq 0$, then $G$ has an element $g$ of order $2$ with fixed points. Thus, $\rho(g)=-\id$.\\ 
    Similarly, if $N_9\neq 0$, there is an element $h\in G$ of order $9$ such that $\rho(h)$ is similar to $\diag(\ze_9,\ze_9^4,\ze_9^7)$. Hence, $\rho(h^3)=\ze_3\cdot\id$.\\
    In these cases, we can compute the possible orders of $G$ from the list in Corollary~\ref{cor:BasketsSing} and the above lemma, which results in finitely many possible groups. Not all of them allow a rigid action on $T$ with the properties of Theorem~\ref{theo:groups}, which we recall below:
\end{Remark}

\begin{Notation}\label{not:Standard}
    In the following, we shall say that a group $G$ enjoys the \emph{standard conditions}, if and only if, for all $g\in G$, it holds
	\[
		\ord(g)\in\{1,2,3,4,5,6,7,8,9,10,12,14\},
	\]
	and there is a 3-dimensional representation $\rho\colon G\to\GL(3,\CC)$ such that:
	\begin{itemize}
		\item $\rho$ is faithful (the action contains no translations),
		\item its character $\chi$ contains no complex conjugated irreducible characters (the action is rigid),
		\item for each $g\in G$, the characteristic polynomial of $\rho(g)\oplus\overline{\rho}(g)$ has integer coefficients ($\rho(g)$ maps the lattice of the torus to itself),
		\item if $\ord(g)\in\{5,8,10,12\}$, then $1\in\Eig(\rho(g))$ (these elements have to act freely), and
        \item if $\ord(g)\in\{7,9,14\}$, then $1\notin\Eig(\rho(g))$ (the fixed points are isolated).
	\end{itemize} 
\end{Notation}

In the sequel, we will frequently use the following version of Burnside's-Lemma counting singularities in two different ways. 

\begin{Lemma}\label{le:CountingFixedPoints}
    Let $T$ be a complex torus, $G$ a finite group acting holomorphically on $T$ such that all stabilizer groups are cyclic. Let $m$ be a divisor of $\lvert G\rvert$.
    Assume that the order of each non-trivial element of $G$ having fixed points is not a proper multiple of $m$. Let $s_m$ be the number of elements of $G$ of order $m$ acting with fixed points, and $\ell$ the number of fixed points of such an element.
   Then
    \[
        \#\{[x]\in T/G\mid \Stab(x)\simeq\ZZ_m\}\cdot\frac{\lvert G\rvert}{m}=\ell\cdot \frac{s_m}{\varphi(m)}
    \]
\end{Lemma}

\begin{proof}
     The left hand side of the equation counts the number of points in $T$ with stabilizer isomorphic to $\ZZ_m$. Each stabilizer contains $\varphi(m)$ elements of order $m$, all of them having the same fixed points. Moreover, generators of different stabilizer groups have disjoint sets of fixed points by the \qq{maximality} of $m$. Thus, the claim follows.
\end{proof}

\begin{Proposition}\label{prop:-id}
    If $-\id\in\im(\rho)$, then $k=9$ in Corollary~\ref{cor:BasketsSing} and $G\simeq \ZZ_{14}$. In particular, the cases $k=1,\ldots,8,10,11$ can not occur.
\end{Proposition}\label{prop:Case1}

\begin{proof}
    By Remark~\ref{rem:centralElements}, $-\id\in\im(\rho)$ holds if and only if $k\in\{1,\ldots,11\}$. By Lemma~\ref{le:Lefschetz}, the number of singularities of even order determine uniquely the group order. They are displayed in the following table:
	\begin{center}
		\bgroup\def\arraystretch{1.3}\begin{tabular}{c||c|c|c|c|c|c|c|c|c|c|c}
			$k$ & $1$ & $2$ & $3$ & $4$ & $5$ & $6$ & $7$ & $8$ & $9$ & $10$ & $11$\\ \hline
			$\lvert G\rvert$ & $96$ & $96$ & $32$ & $56$ & $\tfrac{224}{9}$ & $24$ & $24$ & $16$ & $14$ & $\tfrac{32}{3}$ & $8$
		\end{tabular}\egroup
	\end{center}
    Obviously, the cases $k=5$ and $k=10$ are not possible.
	
	If $k=1$ or $k=6$, then the group order is not divisible by $9$. Hence, the group $G$ doesn't contain any element of order $9$ -- a contradiction to $N_9\neq 0$.\\

	If $k=3,\ 8, \ 11$, then $G$ is a $2$-group of order $8,16$ or $32$. Non of these groups fulfills the standard conditions and the following constrains: if $\lvert G\rvert =8$, then we are in case $k=11$ and $N_4=0$ and thus, also the elements of order $4$ act freely. In the other two cases, all elements of order $4$ having linear parts without eigenvalue $1$ have as set of eigenvalues $\{i,-i\}$ (the fixed points of elements of order $4$ lead to singularities of type $\tfrac{1}{4}(1,1,3) $).\\
	
	If $k=2$ or $k=7$, then $N_6=1$. By Lemma~\ref{le:CountingFixedPoints}, $G$ has precisely $2\cdot \lvert G\rvert/6$ elements of order $6$ whose set of eigenvalues is $\{\ze_6,\ze_6^5\}$. Note that all other elements of order $6$ have the eigenvalue $1$. For the case $k=7$, we observe furthermore that all elements of order $4$ act freely, as $N_4=0$ and stabilizer groups of order $4s$, where $s\geq 2$, do not occur (cf. Theorem~\ref{theo:Singularities}). No group of order $96$ and $24$, enjoys both, the standard and these additional conditions.\\
 
 	If $k=4$, then $\lvert G\rvert=56$ and $N_{14}=2$. By Lemma~\ref{le:CountingFixedPoints}, $G$ has $2\cdot 56\cdot 6/14=48$ elements of order $14$. Note that $G$ has at least six elements of order $7$ and a $2$-Sylow subgroup of order $8$, whose elements have order dividing $8$. Therefore, $G$ can have at most $56-6-8=42$ elements of order $14$ -- a contradiction.\\
	
	If $k=9$, then $\lvert G\rvert =14$ and since $N_{14}\neq 0$, the group $G$ has an element of order $14$. Thus, $G$ is cyclic of order $14$ and the proposition is proven.
\end{proof}

\begin{Proposition}\label{prop:ze3}
    If $\ze_3\cdot\id\in\im(\rho)$, then $k=12$ and $G\simeq \ZZ_9$ or $\ZZ_9\rtimes\ZZ_3$, or $k=15$ and $G\simeq \ZZ_3^2$ or $\ZZ_3^3$. In particular, the cases $k=13$ and $14$ can not occur.
\end{Proposition}

\begin{proof}
    By Proposition~\ref{prop:-id} and Remark~\ref{rem:centralElements}, $-\ze_3\cdot\id$ belongs to the image of $\rho$ if and only if $k\in\{12,\ldots,15\}$. As a consequence of Lemma~\ref{le:Lefschetz}, $N_9$ cannot be $2$. Hence, the case $k=13$ can be excluded.\\
    
	If $k=14$, then $N_9=1$ and so, $N_{3,gor}=0$ and $\lvert G\rvert = 3^5$. By Lemma~\ref{le:CountingFixedPoints}, $G$ has $3^5/9\cdot 6/3=54$ elements of order $9$ and no elements of order greater than $9$. The only groups of order $3^5$ with these properties are the ones with ID $\langle 243,53\rangle$ and $\langle 243, 58\rangle$. Both of these groups do not fulfill the standard conditions -- hence, $k=14$ is also not realizable.\\
 
	If $k=12$, we have $N_9=3$ and thus, $\lvert G\rvert =3^a$ for some $a\in\{2,3,4\}$. If $a=4$, then $N_{3,gor}=0$, and following the same argument as in the case $k=14$, we obtain a contradiction. If $a=3$, then $G$ is either isomorphic to $\ZZ_3\times\ZZ_9$ or to $\ZZ_9\rtimes\ZZ_3$ (we need an element of order $9$), but the first group does not enjoy the standard conditions. If $a=2$, then $G\simeq \ZZ_9$ since $G$ contains an element of order $9$.\\
 
    If $k=15$ then
    \[
	3^3 = N_{3,gor}\cdot \tfrac{\lvert G\rvert}{3}
	\]
    because $N_9=0$. Since $N_3\neq 0$, the order of $G$ is strictly greater than $3$. Note furthermore that $G$ doesn't contain any element of order $9$ because $N_9=0$, i.e., the group has exponent $3$.\\
	Clearly, if $\lvert G\rvert=9$, then $G\simeq\ZZ_3^2$ is the only possibility and if $\lvert G\rvert=27$, then $G\simeq \ZZ_3^3$ (the group $\He(3)$ is not possible since $N_3\neq 0$ and the images of the $3$-dimensional irreducible representations of $\He(3)$ belong to $\SL(3,\CC)$).\\
	The only groups of order $81$ with exponent $3$ are $\ZZ_3^4$ and $\ZZ_3\times \He(3)$. Both groups do not admit a faithful 3-dimensional representation. In case of the first group $\ZZ_3^4$, the representation is the sum of three 1-dimensional characters. Since each of them takes values in  $\langle \ze_3\rangle$, the image of the representation has at  most $3^3$ elements. The second $3$-group $\ZZ_3\times \He(3)$ is not abelian, hence, the representation is irreducible. By Schur's lemma, its center $\ZZ_3^2$ acts by scalar multiples of the identity matrix of order $3$. So, the kernel of the representation is non-trivial.
\end{proof}

In the rest of the section, we consider the remaining case, where the image of $\rho$ does not contain any scalar matrices. Thus we cannot apply Lemma~\ref{le:CountingFixedPoints} to control the group order. Instead, we apply  Sylow's theorems to bound the orders of the $p$-subgroups of $G$. The goal is to prove the following:

\begin{Proposition}\label{prop:LastCase}
    If  $-\id,\ \ze_3\cdot\id\notin\im(\rho)$, then $k=15$ and $G\simeq\ZZ_3^2$.
\end{Proposition}

\begin{Remark}
    By Remark~\ref{rem:centralElements}, it is clear that $k=15$ is the only case, where $-\id,\ \ze_3\cdot\id\notin\im(\rho)$ holds. In this situation, we have the following basket of singularities:
\[
    9\times\tfrac{1}{3}(1,1,2), \quad N_7\times\tfrac{1}{7}(1,2,4).
\]
In particular, $G$ has no elements of order $9$ since such an element would act with fixed points but $N_9=0$.\\
Recall furthermore, that $\lvert G\rvert = 2^a\cdot 3^b\cdot 5^c \cdot 7^d$ with $d\in\{0,1\}$, and $b\geq 1$ because $N_3\ne 0$.
\end{Remark}

\begin{lemma}
	If $G$ has a $5$-Sylow subgroup $S_5$, then $S_5$ is cyclic of order $5$, thus, $c\in\{0,1\}$.
\end{lemma}

\begin{proof}
By Lemma \ref{le:orders} and Sylow's theorem, it suffices to exclude that $G$ contains a copy of  $\mathbb Z_5^2$. 
Assuming the existence of such a subgroup,  the restriction of $\rho$ to $\mathbb Z_5^2$ would be of the form
\[
\rho  \colon \mathbb Z_5^2 \longrightarrow {\rm GL}(3,\mathbb C), \qquad (a,b) \mapsto { \rm diag}(\zeta_5^a, \zeta_5^b,\zeta_5^{\lambda a+ \mu b}), 
\]
up to equivalence of representations and automorphisms of $\mathbb Z_5^2$. All of the representation matrices must have $1$ as an eigenvalue because
the elements of $\mathbb Z_5^2$ can not have fixed points (cf. Lemma \ref{le:orders}). This implies $\lambda=\mu=0$.
Since the characteristic polynomial of $\rho(1,0) \oplus \overline{\rho}(1,0)$ does not have integer coefficients, we observe that an action with such a linear part  cannot exist.
\end{proof}




Furthermore, $G$ has no $7$-Sylow subgroups:

\begin{Lemma}
    The group $G$ has no elements of order $7$, hence, $N_7=d=0$.
\end{Lemma}

\begin{proof}
    Let $n_7$ denote the number of $7$-Sylow subgroups of $G$ and assume that $n_7\geq 1$, so $d=1$. First, we show that $\lvert G\rvert = 7\cdot n_7$. By Lemma~\ref{le:CountingFixedPoints} and since $N_{14}=0$, it holds:
	\[
	N_7\cdot \frac{\lvert G\rvert}{7}=7\cdot n_7.
	\]
	Since $7^2$ doesn't divide the group order, this implies that $N_7$ is divisible by $7$. By Sylow's theorems, there exists an integer $k$ such that $n_7\cdot k=\lvert G\rvert/7$. Hence, $N_7\cdot k=7$, which implies $N_7=7$ and $\lvert G\rvert =7\cdot n_7$.\\
    Note  the group has $n_7\cdot 6$ elements of order $7$ and at least $2/9\cdot \lvert G\rvert$ elements of order $3$ (cf. Lemma~\ref{le:CountingFixedPoints}, recall that $N_{3,gor}=N_9=0$). These are in total more elements than $G$ has:
	\[
		(n_7\cdot 6 + \tfrac{2}{9}\cdot \lvert G\rvert) - \lvert G\rvert = n_7\cdot(6+\tfrac{2}{9}\cdot 7 - 7)=n_7\cdot\tfrac{5}{9} >0. \qedhere
	\]
\end{proof}

\begin{Lemma}\label{le:SylowsCaseIII}
	The $2$-Sylow subgroups contain at most $2^5$ elements, and the $3$-Sylow subgroups are all isomorphic to $\ZZ_3^2$. In particular, $a\leq 5$ and $b=2$.
\end{Lemma}

\begin{proof}
	Let $S_p$ be a $p$-Sylow group of $G$. Then $S_p$ and all its subgroups fulfill the standard conditions except possibly for the rigidity. By Sylow's theorems, $S_p$ has subgroups of order $p^k$ for all $k$ such that $p^k\leq \lvert S_p\rvert$. So, we determine successively by increasing order all possible $p$-groups for $p=2,3$.\\
    If $p=3$, then by assumption, $\ze_3\cdot\id$ is not contained in the image of the representation and the group does not contain elements of order $9$. This excludes all groups of order $3^3$ and the entire list of possible $3$-groups contains only $\ZZ_3$ and $\ZZ_3^2$. By Lemma~\ref{le:CountingFixedPoints}, the number of elements of order $3$ with fixed points equals $2/9\cdot \lvert G\rvert.$ In particular, $9$ divides the group order, so the only possibility for a $3$-Sylow subgroup is $\ZZ_3^2$.\\
	If $p=2$, then for each element $h$ in $S_2$, the matrix $\rho(h)$ has to have eigenvalue 1 since all elements of even order have to act freely. Together with the described strategy, a MAGMA-computation shows that the $2$-subgroups of $G$ have order at most $2^5$.
\end{proof}

\begin{proof}[Proof of Proposition~\ref{prop:LastCase}]
	First, we assume $\lvert G\rvert=2^a\cdot 3^2\cdot 5$ with $a\in\{0,\ldots,5\}$ . The only group having at least $2/9\cdot\lvert G\rvert$ elements of order $3$ and fulfilling that the order of each element of $G$ belongs to $\{1,2,3,4,5,6,8,10,12\}$ has MAGMA-ID $\langle 360, 118\rangle$. But this group is not abelian and has no irreducible character of degree $2$ or $3$ -- a contradiction. Hence, $5$ does not divide the group order and $\lvert G\rvert=2^a\cdot 3^2$.\\
	If moreover $a=0$, then $G$ is isomorphic to $\ZZ_3^2$.\\
	In order to exclude the case $a\geq 1$, we check with MAGMA that there is no group $G$ of order $2^a\cdot 3^2$ with $a\in\{1,\ldots,5\}$ such that the order of each element belongs to $\{1,2,3,4,6,8,12\}$, there are at least $2/9\cdot\lvert G\rvert$ elements of order $3$, and such that $G$ enjoys the standard conditions and additionally
	\begin{itemize}
		\item $\#\{g\in G\mid \ord(g)=3, \ 1\notin\Eig(\rho(g))\}=\tfrac{2}{9}\cdot\lvert G\rvert$ and
		\item if $\ord(g)=3$ and $1\notin\Eig(\rho(g))$, then $\Eig(\rho(g))=\{\ze_3,\ze_3^2\}$. $\hfill\qedhere$
	\end{itemize}
\end{proof}


\section{Classification of the Quotients}\label{sec:classQuot}

The classification of the quotients with $p_g(X)=1$ was already done in \cite{GK}. If $p_g(X)=0$, then by Theorem~\ref{theo:groups}, the group $G$ is one of the following:
\[
    \ZZ_9,\quad \ZZ_{14},\quad \ZZ_3^2,\quad \ZZ_3^3, \quad \ZZ_9\rtimes\ZZ_3.
\]

For the cyclic groups, the situation is easy to handle:

\begin{prop}\label{prop:Cyclic}
	For $G=\ZZ_9$ and $G=\ZZ_{14}$, there exists up to biholomorphism one and only one quotient $X=T/G$.
\end{prop}

\begin{proof}
    Moving the origin, we can assume that $G$ acts linearly with generators
    \[
        \diag(\ze_9,\ze_9^4,\ze_9^7)\qquad \makebox{or} \qquad \diag(\ze_{14},\ze_{14}^9,\ze_{14}^{11}),
    \]
    respectively (cf. Theorem~\ref{theo:Singularities}). This implies that $T$ is of CM-type in each case (cf. \cite{BL}*{Theorem~13.3.2}) and uniquely determined due to \cite{Shimura}*{Proposition~17, p. 60} since the class numbers of the cyclotomic fields $\QQ(\ze_9)$ and $\QQ(\ze_{14})$ are one.
\end{proof}

For the non-cyclic groups, the situation is more involved because it is not possible to assume that the action is linear. Here, we use the \qq{classification machinery} from \cite{GK} outlined in Section~\ref{sec:preliminaries}, and treat the different groups separately.


\subsection{The Case $G=\ZZ_3^3$}\

Up to equivalence of representations and automorphisms of $G=\ZZ_3^3$, the only faithful representation of dimension $3$ of $\ZZ_3^3$ is given by
\[
    \rho\colon \ZZ_3^3\longrightarrow\GL(3,\CC),\quad (a,b,c)\mapsto \diag(\ze_3^a,\ze_3^b,\ze_3^c).
\]

\begin{rem}\label{rem:StructureTforZ3^3}
    Let $T$ be a $3$-dimensional torus admitting an action $\Phi$ of $\ZZ_3^3$ with linear part $\rho$. The subtori 
\[
E_1:=\ker(\rho(0,1,1)-\id)^0,\quad E_2:=\ker(\rho(1,0,1)-\id)^0,\quad E_3:=\ker(\rho(1,1,0)-\id)^0
\]
of $T$ are all isomorphic to Fermat's elliptic curve $E=\mathbb C/ \mathbb Z[\zeta_3]$ since this is the unique elliptic curve where  multiplication by $\zeta_3$ is an automorphism.
The addition map $E_1\times E_2\times E_3\to T$ is an isogeny and induces an equivariant isomorphism $T\simeq  E^3/K$, where $K$ is the kernel.\\
Since $T$ containes $E_j$ as a subtorus, $K$ cannot contain elements of the form $\lambda e_j$ with $\lambda\neq 0$.\\

From now on, we fix the following generators of $G=\ZZ_3^3$:
\[
    k:=(1,\ 1, \ 1), \quad h:=(0,\ 2,\ 1),\quad g:=(1,\ 0, \ 0).
\]
Furthermore, by possibly changing the origin of $T$, we assume that the cocycle $\tau$ is given by
\[
    \tau(k)=0,\quad \tau(h)=(a_1,\ a_2,\ a_3), \quad \tau(g)=(b_1,\ b_2, \ b_3)\quad \in T.
\]
We will refer to such a cocycle as cocycle in \textit{standard form}. Note that
\[
    \ze_3\cdot\tau(g)=\rho(k)\cdot\tau(g)+\tau(k)=\tau(kg)=\tau(gk)=\tau(g).
\]
Thus, $\tau(g)=(b_1,b_2,b_3)$ is fixed by multiplication with $\ze_3$. Analogous, $\tau(h)=(a_1,a_2,a_3)\in\Fix_{\ze_3}(T)$.
\end{rem}

\begin{Lemma}
     Let $\tau$ be a cocycle in standard form. Then the following holds:
    \begin{itemize}
        \item $a_1\in E[3]$,
        \item $(0,3b_2,3b_3)\in K$,
        \item $v:=((\ze_3-1)a_1,(1-\ze_3^2)b_2,(1-\ze_3)b_3)\in K$.
    \end{itemize}
    Conversely, two elements $(a_1,a_2,a_3),(b_1,b_2,b_3)$ in $\Fix_{\ze_3}(T)$ fulfilling these conditions yield a well-defined  cocycle in standard form. 
\end{Lemma}

\begin{proof}
    The corresponding action has to be a group homomorphism, so it has to preserve the relations of the elements in $G=\ZZ_3^3$. This leads to:
    \begin{itemize}
        \item $\tau(h^3)=0 \iff a_1\in E[3]$.
        \item $\tau(g^3)=0\iff (0,3b_2,3b_3)\in K$.
        \item $\tau(gh)=\tau(hg)\iff v\in K$.
    \end{itemize}
    Note that $\tau(hk)=\tau(kh)$ and $ \tau(gk)=\tau(kg)$ is always fulfilled since $\tau(h),\tau(g)\in\Fix_{\ze_3}(T)$. Moreover, $\tau(k^3)=0$ since $\tau(k)=0$.
\end{proof}

We will call a cocycle $\tau$ in standard form \emph{good}, if the corresponding action has only isolated fixed points. The latter is the case if all elements whose linear parts of the action have 1 as eigenvalue act freely. Since all non-trivial elements in $\ZZ_3^3$ have order $3$, the elements $u$ and $u^2$ have the same fixed points. Thus, the action has isolated fixed points if and only if the elements $h,hk,hk^2,g,ghk^2,gh^2k^2,gk^2,ghk,gh^2k$ act freely. This leads to the following conditions on the cocycle:

\begin{Lemma}\label{le:good}
    A cocylce in standard form is good if and only if the following conditions are satisfied:
    \begin{enumerate}
        \item For all $i=1,2,3$: $a_i$ is never the $i$-th coordinate of an element in $K$.
        \item There are no elements in $K$ of the forms $$(\ast,b_2,b_3),\quad (\ze_3a_1+b_1,\ast,a_3+b_3), \quad(2\ze_3a_1+b_1,-\ze_3a_2+b_2,\ast).$$
        \item $b_1$ is never the first coordinate of an element in $K$.
        \item $a_2+b_2$ is never the second coordinate of an element in $K$.
        \item $-\ze_3^2a_3+b_3$ is never the third coordinate of an element in $K$.
    \end{enumerate}
\end{Lemma}

\begin{proof}
    We only consider the element $h$. The computation of the other ones is similar.
    An element $z=(z_1,z_2,z_3)\in T$ is a fixed point of $\Phi(h)$ if and only if
    \[
        (a_1, (\ze_3^2-1)z_2+a_2,(\ze_3-1)z_3+a_3) \in K.
    \]
    Thus, we see that $\Phi(h)$ has a fixed point if and only if there is an element $(a_1,\ast,\ast)\in K$.
\end{proof}

Restricting  $\rho$ to the subgroup $\langle h,k\rangle$ of $G$, we obtain the analytic representation of $\ZZ_3^2$ studied in \cite{GK}. Thus, we have:

\begin{Lemma}[\cite{GK}*{Proposition~4.7}]
    The kernel $K$ of the addition map $E_1\times E_2\times E_3\to T$ is contained in $\Fix_{\ze_3}(E)^3\simeq \ZZ_3^3$.
\end{Lemma}

\begin{Proposition}[\cite{GK}*{Proposition~4.19}]\label{prop:normalizer}
    Any biholomorphism $f\colon X\to X'$ between two quotients with group $\ZZ_3^3$, where the actions have only isolated fixed points, is induced by a biholomorphic map $$\hat{f}\colon E^3\to E^3, \quad z\mapsto Cz+d,$$ such that $CK=K'$. This means that $C$ is contained in the normalizer $\sN:=N_{\Aut(E^3)}(\rho(\ZZ_3^3))$, which is a finite group of order $6^4=1296$ and generated by the matrices
    \[
        \begin{pmatrix}
            -\ze_3 & & \\ & 1 & \\ & & 1
        \end{pmatrix}, \quad
        \begin{pmatrix}
            0 & 1 & 0\\ 0 & 0 & 1 \\ 1 & 0 & 0
        \end{pmatrix}\quad \makebox{and}\quad
        \begin{pmatrix}
            0 & 1 & 0 \\ 1 & 0 & 0 \\ 0 & 0 & 1
        \end{pmatrix}.
    \]
\end{Proposition}

\begin{rem}
    According to Proposition~\ref{prop:normalizer}, the normalizer $\sN$ acts on the set of kernels $\sK$. In particular, it is enough to consider one representative of each orbit to find all biholomorphism classes of quotients. Furthermore, quotients of tori with kernels of different orbits are never biholomorphic. The orbits are represented by
    \[ 
        K_1:=\{0\}, \quad K_2:=\langle(t,t,0)\rangle, \quad K_3:=\langle (t,t,t)\rangle,\quad K_4:=\langle(t,t,t),(t,-t,0)\rangle.
    \]
\end{rem}

\begin{rem}\label{rem:ActionsZ3^3}
    Let $\tau_1,\tau_2$ be two cocycles in standard form. Then they belong to the same classes in $H^1(\ZZ_3^3,T)$ if and only if there exists a $d\in T$ such that $\rho(u)\cdot d- d=\tau_1(u)-\tau_2(u)$ for all $u\in\ZZ_3^3$. Evaluating this equation in $k$ tells us that $d$ has to belong to $\Fix_{\ze_3}(T)$.
By running a MAGMA implementation, we can determine all actions and classes of good cocycles of $\ZZ_3^3$ on $T=E^3/K_i$.
\begin{center}
	\bgroup\def\arraystretch{1.3}\begin{tabular}{|c|l|l| l |} \hline
		$i$  & $K_i$  &   \# of actions & \# of good classes in $H^1(\mathbb Z_3^3, E^3/K_i)$  \\ 
		\hline \hline
		$1$ & $\{0\}$ & $16$ & $16$\\ \hline
		$2$ & $\langle (t,t,0)\rangle$ & $48$ & $16$ \\
		$3$ & $\langle (t,t,t) \rangle$ & $0$ & $0$ \\ \hline
		$4$ & $\langle (t,t,t), (t,-t,0)\rangle$ & $0$ & $0$ \\ \hline
	\end{tabular}\egroup
\end{center}
In particular, there are no actions with isolated fixed points on the tori $E^3/K_3$ and $E^3/K_4$.
\end{rem}

\begin{rem}\label{rem:Cd}
By Proposition~\ref{prop:normalizer} and Proposition~\ref{prop:classes}, the group of potential linear parts of biholmorphisms of quotients of $E^3/K_i$ is
\[
    \sN_\CC(\Lam_{K_i})=\{C\in\sN\mid C\cdot K_i=K_i\}.
\]
Recall that $\tau$ and $\tau'$ lead to biholomorphic quotients if and only if there exists a matrix $C\in \sN_\CC(\Lam_{K_i})$  and an element $d\in T=E^3/K_i$ such that
\[ (\mathrm{a})\: C\rho(u)C^{-1}=\rho(\varphi_C(u)) \qquad \mathrm{and}\qquad (\mathrm{b})\: (\rho(u)-\id)d=C\tau(\varphi_C^{-1}(u))-\tau'(u) \]
for all $u\in \ZZ_3^3$. Since $\rho(k)=\ze_3\cdot\id$ and $\rho$ is faithful, item (a) implies that $\varphi_C(k)=k$. Hence, $(\rho(k)-\id)d=0$ by item (b), so $d\in\Fix_{\ze_3}(T)$.
\end{rem}

\begin{Proposition}\label{prop:ClassificationZ3^3}
    There are precisely 3 biholomorphism classes of rigid quotients of 3-dimensional tori by rigid actions of $\ZZ_3^3$ with isolated fixed points. More precisely:
    \begin{itemize}
        \item    One of them is realized as a quotient of $T_1=E^3/K_1$ and corresponds to $Y_9$ of Theorem~\ref{theo:main}.
    The other two classes are realized as quotients of $T_2=E^3/K_2$ and correspond to $Y_{10}$ and $Y_{10'}$.
    \item The quotients $Y_{10}$ and $Y_{10'}$ are diffeomorphic to each other but not diffeomorphic to $Y_9$.
    \end{itemize}
\end{Proposition}

\begin{proof}
    The biholomorphism classes of the quotients corresponds to the orbits of the action of $\sN_\CC(\Lam_{K_i})$ on the good cohomology classes in $H^1(G,E^3/K_i)$ (cf. Proposition~\ref{prop:classes}). Using the excplicit despriction of $\sN_\CC(\Lam_{K_i})$ and the coboundaries in Remark~\ref{rem:Cd},
    this computation is done by a MAGMA implementation. The diffeomorphism
    \[
        C\colon \CC^3\longrightarrow \CC^3, \quad (z_1,z_2,z_3)\mapsto (\overline{z_1},\overline{z_2},\overline{z_3}),
    \]
    induces a diffeomorphism between the named quotients of $T_2$.\\
    Since $H^0(\ZZ_3^3,T_1)\simeq \ZZ_3^3$ and $H^0(\ZZ_3^3,T_2)\simeq\ZZ_3^2$, quotients of $T_1$ cannot be diffeomorphic to quotients of $T_2$ by Proposition~\ref{prop:classes}.
\end{proof}


\subsection{The Case $G=\ZZ_3^2$}

\begin{Proposition}\label{prop:AnalyticRepr}
	The analytic representation of a rigid, faithful and translation-free action of the group $\ZZ_3^2$ on a 3-dimensional complex torus leading to a quotient $X$ with $p_g(X)=0$ is up to an automorphism of $\ZZ_3^2$ equivalent to one of the two representations 
	\[
        \rho_1(a,b)=\diag(\ze_3^a,\ze_3^b,\ze_3^{ a+b})\qquad \makebox{or}\qquad \rho_2(a,b)=\diag(\ze_3^a,\ze_3^a,\ze_3^b).
    \]
\end{Proposition}

\begin{proof}
	Let $\rho\colon \ZZ_3^2\hookrightarrow \GL(3,\CC)$ be such a representation. Write $\rho=\diag(\chi_1,\chi_2,\chi_3)$ with characters $\chi_j$ of degree 1. Since $\rho$ is faithful and the action is rigid, we can assume that $\chi_1$ and $\chi_2$ are linearly independent in the group of characters. In other words
 \[
    \rho(a,b)=\diag(\ze_3^a,\ze_3^b,\ze_3^{\lambda a+\mu b} )\qquad\makebox{for some}\: \lambda,\mu\in \{0,1,2\}.
 \]
    Due to the rigidity and $p_g(X)=0$, it holds $\rho=\rho_1,\ \rho_2$ or $\rho_3(a,b)=\diag(\ze_3^a,\ze_3^b,\ze_3^{a+2 b})$, up to a permutation of coordinates. Twisting $\rho_1$ by the automorphism $(a,b)\mapsto (a,2a+b)$ of $\ZZ_3^2$ gives a representation equivalent to $\rho_3$.
\end{proof}

\underline{The subcase $\rho=\rho_1$:}

For analyzing the situation where $\rho=\rho_1$, we choose the generators $h=(0,2)$ and $k=(1,1)$ for $\ZZ_3^2$. Note that the representation $\rho_1$ is obtained from the representation of $\ZZ_3^2$ in \cite{GK} by conjugating the third coordinate. We may assume that the translation part of $k$ is $\tau(k)=0$.
The classification of the quotients can be done in analogy to \cite{GK} and to the case $G=\ZZ_3^3$ above, so we only state the result:

\begin{Proposition}\label{prop:oneBihol}
    There are precisely 5 biholomorphism classes of rigid quotients of 3-dimensional tori by rigid actions of $\ZZ_3^2$  with analytic representation $$\rho_1(a,b)=\diag(\ze_3^a,\ze_3^b,\ze_3^{a+b}).$$ They are represented by $Y_3,Y_4,Y_{4'},Y_5$, and $Y_6$ of Theorem~\ref{theo:main}. The quotients $Y_4$ and $Y_{4'}$ are diffeomorphic via $(z_1,z_2,z_3)\mapsto (-z_1,\overline{z_3},\overline{z_2})$. All other quotients are pairwise not diffeomorphic.
\end{Proposition}

\underline{The subcase $\rho=\rho_2$:}

It remains to analyze the case, where the analytic representation is given by $\rho=\rho_2$. Here, the situation is different since $\rho_2$ contains two times the same character. This time, we choose $h:=(1,0)$ and $k:=(1,1)$ as generators of $\ZZ_3^2$.

\begin{rem}\label{rem:StructureTforZ3^2}
Consider the subtori $E_3:=\ker(\rho(h)-\id_T)^0$ and $T':=\ker(\rho(h^2k)-\id_T)^0$ of $T$. Then the addition map
\[
	\mu\colon T'\times E_3\longrightarrow T
\]
defines an equivariant isogeny. As $\ze_3$ acts on $E_1$, this curve is isomorphic to $E=\CC/\ZZ[\ze_3]$. The action of $\rho$ restricted to $T'$ is given by $\ze_3^a\cdot\id_{T'}$, hence, $T'$ is equivariantly isomorphic to $E^2$ (cf. \cite{BL}*{Corollary~13.3.5}).\\
 In summary, we have that $T$ is equivariantly isomorphic to $E^3/K$ and the maps
 \[
 	E_3\hookrightarrow T=(T'\times E_3)/K \qquad \makebox{and}\qquad T'\hookrightarrow T=(T'\times E_3)/K 
 \]
 are injective. In particular, we can assume without loss of generality that $T$ is of the form $E^3/K$, where $K$ is finite and does not contain non-zero elements of the form $\lambda e_3$ or $\mu e_1+\tau e_2$.
\end{rem}
 
\begin{Lemma}
	The kernel $K$ of the addition map is contained in $\Fix_{\ze_3}(E)^3$.
\end{Lemma}

\begin{proof}
	Let $(t_1,t_2,t_3)\in K$. For $u\in \ZZ_3^2$, we view $\rho(u)$ as an automorphism of $E^3=T'\times E_3$ mapping $K$ to itself. For $u=h$ and $h^2k$, we therefore get that the elements
	\begin{align*}
		(\rho(h)-\id_T)(t)&=((\ze_3-1)t_1,(\ze_3-1)t_2,0)\qquad \makebox{and}\\
		(\rho(h^2k)-\id_T)(t)&=(0,0,(\ze_3-1)t_3)
	\end{align*}
	belong to $K$. This implies that $((\ze_3-1)t_1,(\ze_3-1)t_2)=0$ in $T'=E^2$ and $(\ze_3-1)t_3=0$ in $E_3=E$.  Thus, $(t_1,t_2,t_3)\in \Fix_{\ze_3}(E)^3$.
\end{proof}

Let $\Phi\colon \ZZ_3^2\hookrightarrow \Bihol(T)$ be a faithful action with analytic representation $\rho=\rho_2$. Then, up to a change of the origin in $T$, the translation part $\tau\colon \ZZ_3^2\to T$ can be written as 
\[
\tau(h)= \left( a_1, \ a_2, \ a_3\right), \quad
\tau(k)= \left( 0, \ 0, \ 0\right).
\]
As above, we say that the cocycle is in \emph{standard form}. Since $\tau(h^3)$ and $\tau((hk^2)^3)$ are zero in $T$, the elements $a_1,a_2$ and $a_3$ belong to $E[3]$. Using furthermore that $\tau(hk)=\tau(kh)$, we obtain:

\begin{Lemma}\label{le:welldefinedZ3^2 rho3}
	Let $\tau$ be a cocycle in standard form. Then the vector $$v:=(\ze_3-1)\cdot (a_1,\ a_2,\ a_3)$$ is zero in $T$, i.e. $(a_1,a_2,a_3)$ is fixed by $\ze_3\cdot \id$. Conversely, given $a_1,a_2,a_3\in E[3]$ such that $v=0$ in $T$, we obtain a cocycle $\tau\colon \ZZ_3^2\to T$ in standard form.
\end{Lemma}

\begin{Lemma}\label{le:good Z3^2 rho}
	A cocycle $\tau\colon \ZZ_3^2\to T$ in standard form is good, i.e., the corresponding action has isolated fixed points,  if and only if  $K$ contains no elements of the form $(\ast,\ast,a_3)$ or $(a_1,a_2,\ast)$.
\end{Lemma}

\begin{proof}
     The corresponding action has isolated fixed points if and only if $h$ and $hk^2$ act freely. We conclude as in the proof of Lemma~\ref{le:good}.
\end{proof}

Similar to \cite{GK}*{Proposition~4.19}, we see:

\begin{Proposition}\label{prop:norm}
	Every biholomorphism $f\colon X\to X'$ between two quotients by $\ZZ_3^2$, where the actions have isolated fixed points and linear part $\rho_2$, is induced by a biholomorphic map $$\hat{f}\colon E^3\to E^3,\:z\mapsto Cz+d,$$ such that $CK=K'$. This means that $C$ is contained in the normalizer group
	\[
		\sN:=N_{Aut(E^3)}(\rho(\ZZ_3^2))=\left\{ \begin{pmatrix}
			C' & 0\\0 &c 
		\end{pmatrix}\big\vert\, c\in \langle -\ze_3\rangle,\: C'\in \GL(2,\ZZ[\ze_3])\right\}.
	\]
	In particular, it holds:
	\[
		\sN_\CC(\Lambda_K)=\{C\in N_{\Aut(E^3)}(\rho(\ZZ_3^2))\mid CK=K\}.
	\]
\end{Proposition}

\begin{Remark}
	Note that in this situation, the normalizer group is infinite. Nevertheless, it acts on the finite set of potential kernels $\sK$ and it suffices to consider one kernel in each orbit. Furthermore, quotients of tori with kernels in different orbits cannot be biholomorphic.
\end{Remark}

\begin{Lemma}\label{le:kernelsZ3^2}
    The kernel $K$ is either trivial or belongs to the orbit of $K_1:=\langle(t,t,t)\rangle$.
\end{Lemma}

\begin{proof}
	Each 2-dimensional subspace of $\Fix_{\ze_3}(E)^3$ contains a non-trivial  element of the form $\mu e_1+ \tau e_2$ and can therefore be excluded as a kernel. The only 1-dimensional subspaces without such elements are $\langle(t,0,t)\rangle$,  $\langle(0,t,t)\rangle$ and  $\langle(t,t,t)\rangle$. 
    The first and the second are mapped to the third by
	\[
		\begin{pmatrix}
			1 & 0 & \\ 1& 1 & \\ &&1
		\end{pmatrix}\qquad \makebox{and} \qquad \begin{pmatrix}
		 1 & 1 & \\ 0 & 1 & \\ &&1
	\end{pmatrix},
	\]
    respectively. Thus, they all belong to the $\sN$-orbit of $\langle(t,t,t)\rangle$.
\end{proof}

\begin{Proposition}\label{prop:Classification2Z3^2}
    There are precisely 2 biholomorphism classes of rigid quotients of 3-dimensional tori by rigid actions of $\ZZ_3^2$  with analytic representation $$\rho_2(a,b)=\diag(\ze_3^a,\ze_3^a,\ze_3^b).$$ For each torus $T_0=E^3/K_0$ and $T_1=E^3/K_1$, there is one class, they are represented by $Y_7$ and $Y_8$ of Theorem~\ref{theo:main} and are not diffeomorphic.
\end{Proposition}

\begin{proof}
    Since $H^0(\ZZ_3^2,T_0)\simeq \ZZ_3^3$ and $H^0(\ZZ_3^2,T_1)\simeq\ZZ_3^2$, a quotient of $T_0$ cannot be diffeomorphic to a quotient of $T_1$ by Proposition~\ref{prop:classes}.\\
    Next, we prove that for both tori, the normalizer $\sN_\CC(\Lambda_{K_i})$ acts transitively on the set of good cocycles.
    Since all matrices $C\in N_{\Aut(E^3)}(\rho(\ZZ_3^2))$ are in block-form $C=\diag(C',c)$ with $C'\in\GL(2,\ZZ[\ze_3])$ and $c\in\langle -\ze_3\rangle$ by Proposition~\ref{prop:norm}, they commute with $\rho$. Hence, it suffices to show that for any two good cocylces $\tau$ and $\tau'$ in standard form, there exists a matrix $C\in\sN_\CC(\Lambda_{K_i})$ such that $C\cdot\tau=\tau'$.  Evaluated in $k$, this equation automatically holds since $\tau(k)=\tau'(k)=0$. So it suffices to check that $C\cdot \tau(h)=\tau'(h)$.\\
    We start with $T_0=E^3$. Let $\tau(h)=(a_1,a_2,a_3)$. Then by the Lemmata~\ref{le:welldefinedZ3^2 rho3} and \ref{le:good Z3^2 rho}, the cocylce $\tau$ is good if and only if $a_i\in\Fix_{\ze_3}(E)=\{0,\pm t\}$, $a_3\neq 0$ and $(a_1, a_2)\neq (0,0)$. 
    Let $\tau(h):=(t,0,t)$ and $\tau'(h)=(a'_1,a'_2,a'_3)$ be an arbitrary cocycle. Choose $c\in\{\pm 1\}$ such that $c\cdot t=a'_3$. Then a suitable matrix 
    \[
		C'\in \left\{\begin{pmatrix}
			\pm 1 & 0 \\ \pm 1 & 1
		\end{pmatrix},\:  \begin{pmatrix}
		0 &  1 \\  \pm 1 & 0
	\end{pmatrix}, \:  \begin{pmatrix}
	    \pm 1 & 0 \\ 0 & 1
	\end{pmatrix}\right\}
	\]
    yields $C:=\diag(C',c)\in \sN_\CC(\Lambda_{K_i})$ with $C\cdot(t,0,t)=\tau'(h)$.\\	
	Finally, we consider the quotients of $T_1=E^3/K_1$. There are six good cohomology classes in $H^1(\ZZ_3^2,T_1)$ represented by
	\begin{center}
		\bgroup\def\arraystretch{1.3}\begin{tabular}{|c|cccccc|}
			\hline
			$i$ & 1 & 2 & 3 & 4 & 5 & 6  \\
			\hline
			$\tau_i(h)$ & $\frac{1}{3}\cdot\footnotesize{\begin{pmatrix} 1 \\1 \\ 1 \end{pmatrix}}$ & $\frac{1}{3}\cdot \footnotesize{\begin{pmatrix} \ze_3 \\ 1 \\ 1 \end{pmatrix}}$ & $\frac{1}{3}\cdot \footnotesize{\begin{pmatrix} 1 \\ \ze_3 \\ 1 \end{pmatrix}}$  &  $\frac{2}{3}\cdot \footnotesize{\begin{pmatrix} 1 \\1 \\ 1 \end{pmatrix}}$ & $\frac{2}{3}\cdot \footnotesize{\begin{pmatrix} \ze_3 \\ 1 \\ 1 \end{pmatrix}}$ & $\frac{2}{3}\cdot \footnotesize{\begin{pmatrix} 1 \\ \ze_3 \\ 1 \end{pmatrix}}$ \\
			\hline
		\end{tabular}\egroup
	\end{center}
    We finally give matrices $C_{ij}\in \sN_\CC(\Lambda_{K_1})$ such that $0 = C_{ij}\cdot \tau_i(h)-\tau_j(h)$ in $T_1$:

    \[
        C_{ij}:=\begin{cases}
                \diag(\ze_3,1,1), & \makebox{if}\: (i,j)\in\{(1,2),(4,5)\}, \\
                \diag(1,\ze_3,1), & \makebox{if}\: (i,j)\in\{(1,3),(4,6)\},\\
                -\id, & \makebox{if}\: (i,j)=(1,4)
            \end{cases}.
    \]
    Hence all classes belong to the same orbit.
\end{proof}


\subsection{The Case $G=\ZZ_9\rtimes\ZZ_3$}

\begin{Proposition}\label{prop:ClassificationZ9:Z3}
    There is one and only one biholomorphism class of rigid quotients of 3-dimensional tori by rigid actions of $$\ZZ_9\rtimes\ZZ_3=\langle g,h\mid h^3=g^9=1,\: hgh^{-1}=g^4\rangle$$ with isolated fixed points, which can be realized as a quotient of $T=E^3$ and corresponds to $Y_{11}$ of Theorem~\ref{theo:main}.
\end{Proposition}

\begin{proof}
    Let $T$ be a $3$-dimensional torus admitting an action $\Phi$ of $\ZZ_9\rtimes \ZZ_3$ as in the proposition. Then up to automorphisms of $\ZZ_9\rtimes \ZZ_3$ and equivalence of representations, the anyalytic representation $\rho\colon \ZZ_9\rtimes \ZZ_3\hookrightarrow\GL(3,\CC)$ is given by
    \[
        \rho(g)=\begin{pmatrix}
            0 & 1 & 0\\ 0& 0 & 1\\\ze_3 & 0 & 0
        \end{pmatrix}\qquad \makebox{and}\qquad \rho(h)=\begin{pmatrix}
            \ze_3 & & \\ & 1 &\\ &&\ze_3^2
        \end{pmatrix}.
    \]
    The subtori
    \[
        E_1:=\ker(\rho(hg^6)-\id)^0,\quad E_2:=\ker(\rho(h)-\id)^0,\quad E_3:=\ker(\rho(hg^3)-\id)^0
    \]
    of $T$ are all isomorphic to Fermat's elliptic curve $E$ and the addition map $E_1\times E_2\times E_3\to T$ is an isogeny inducing an equivariant isomorphism $T\simeq E^3/K$.\\
    Without loss of generality, we will assume that the translation part $\tau$ of $\Phi$ fulfills:
    \[
        \tau(h)=(a_1,a_2,a_3)\qquad \makebox{and}\qquad \tau(g)=0.
    \]
    Then we see as before that $\Phi$ is well-defined if and only if $a_1,a_2$ and $a_3$ are $3$-torsion points in $E$ and the element $\tau(h)$ is fixed by $\rho(g)$. Furthermore, the action has isolated singularities if and only if $h, hg^3$ and $hg^6$ act freely, which is equivalent to say that $a_i$ is never the $i$-th coordinate of an element in $K$. Again similar as in the other cases, we deduce that $K$ is a subgroup of $\Fix_{\ze_3}(E)^3$ containing no elements of the form $\lambda e_j$ with $\lambda\neq 0$. Furthermore, $K$ has to be fixed under multiplication by $\rho(g)$. Thus, $K$ is one of the following:
    \[
        K_0:=\{0\},\qquad K_1:=\langle (t,t,t)\rangle, \qquad K_2:=\langle (t,t,t),(t,-t,0)\rangle,
    \]
    where $t:=1/3+2\ze_3/3$. A MAGMA computation shows that there are no actions with isolated fixed points if $K=K_1$ or $K=K_2$ and that the only two possible actions for $K=K_0$, thus $T=E^3$, are given by
    \[
        \tau_1(h)=t\cdot (1,1,1)\qquad \makebox{and}\qquad \tau_2(h)=-t\cdot (1,1,1).
    \]
    Since multiplication by $-1$ induces a biholomorphism between the corresponding quotients, the claim follows.
\end{proof}


\section{Fundamental Groups and Covering Spaces of  Torus-Quotients}\label{sec:FundGroups}

Let $G$ be a finite group acting holomorphically on a complex torus $T$ of arbitrary dimension by
\[
    \Phi(g)(z)=\rho(g)\cdot z + \tau(g).
\]

The goal of this section is to study the structure of the fundamental group and the covering space of the torus-quotient $X=T/G=\CC^n/\Gamma$, where $\Gamma$ is the orbifold fundamental group
\[
    \Gamma:=\pi_1^{orb}(T,G)=\{\gamma\colon \CC^n\to\CC^n\mid \exists \:g\in G \colon g\circ p =   p \circ \gamma\}
\]
sitting inside the exact sequence
\[
    0\longrightarrow \Lambda\longrightarrow \Gamma \overset{\pi}{\longrightarrow} G\longrightarrow 1
\]
as explained in Section~\ref{sec:preliminaries}. Applying the results to the quotients in our classification will allow us to compute their fundamental groups explicitly.

\begin{Notation}
    We denote by $G_{\fix}$ and $\Gamma_{\fix}$ the subgroups of $G$ and $\Gamma$ generated by the elements acting with fixed points on $T$ and $\CC^n$, respectively. These subgroups are normal.
\end{Notation}

By a theorem of Armstrong (cf.~\cite{Armstrong}), it holds:
    \[
        \bigslant{\Gamma}{\Gamma_{\fix}}\simeq \pi_1\left(\bigslant{T}{G}\right).
    \]

Since the quotient map $\pi\colon\Gamma\to G$ restricts to a surjection $\pi\colon \Gamma_{\fix} \to G_{\fix}$ with kernel $\Lambda\cap \Gamma_{\fix}$, the following sequence is exact:
\begin{align}\label{eq:fundamentalgroup}
    0 \longrightarrow \bigslant{\Lambda}{(\Lambda\cap\Gamma_{\fix})} \longrightarrow\bigslant{\Gamma}{\Gamma_{\fix}}\longrightarrow\bigslant{G}{G_{\fix}}\longrightarrow 1.
\end{align}

\begin{Corollary}\label{cor:universalCover}
    The universal cover of $T/G$ is 
    \[
    \CC^n/\Gamma_{\fix} \simeq \bigslant{(\CC^n/\Lambda')}{(\Gamma_{\fix}/\Lambda')}\simeq \bigslant{T'}{G_{\fix}},
    \]
    where $\Lambda':=\Lambda\cap\Gamma_{\fix}$ and $T'=\CC^n/G_{\fix}$ is a possibly non-compact torus.\\
    In particular, the universal cover of $T/G$ is compact if and only if $\Gamma_{\fix}$ is crystallographic, or equivalently $T'$ is compact.
\end{Corollary}

\begin{Corollary}\label{cor:fundamentalgroup}
    If there is an element $g\in G_{\fix}$ such that all lifts of $g$ to $\CC^n$ belong to $\Gamma_{\fix}$, then $\Lambda\subset\Gamma_{\fix}$ and in particular,
    \[
        \pi_1\left(\bigslant{T}{G}\right)\simeq \bigslant{G}{G_{\fix}}.
    \] 
\end{Corollary}

\begin{rem}\label{rem:fundamentalgroup}
    If $G$ contains an element $g$ such that 1 is not an eigenvalue of $\rho(g)$, then $g\in G_{\fix}$ and all lifts of $g$ have a fixed point. Hence, the condition $\Lambda\subset \Gamma_{\fix}$ of Corollary~\ref{cor:fundamentalgroup} is satisfied.\\
    In all our examples, there is such an element, which allows us to compute the fundamental group and the universal cover immediately.
\end{rem}

\begin{rem}
In the general case, where $ \Lambda \neq (\Lambda\cap\Gamma_{\fix})$, we still have a description of the fundamental group in terms of the $G$-action as an extension:     \[
        \pi_1\left(\bigslant{T}{G}\right)\simeq \bigslant{\Lambda}{(\Lambda\cap\Gamma_{\fix})}\times_{\overline{\beta}}\: \bigslant{G}{G_{\fix}},
    \]
    where the $2$-cocyle $\overline{\beta}$ can be derived from the translation part $\tau$ of the action on the torus.  
For completeness, we sketch the construction: 

    
    The short exact sequence  $0\to \Lambda \to \mathbb C^n \to T \to 0$ of $G$-modules yields an isomorphism 
    \[
    \delta\colon H^1(G,T)\longrightarrow H^2(G,\Lambda).
    \]
    The image of the class of  $\tau$ under this isomorphism yields the crystallographic group $\Gamma$ as an extension of $G$ by $\Lambda$.\\
    We describe how to compute a representative $\beta \in Z^2(G,\Lambda)$ of the image of $\tau$ which descends to a $2$-cocyle $\overline{\beta}$ of $G/G_{\fix}$ with values in $\Lambda/(\Lambda\cap\Gamma_{\fix})$ giving the extension
    \[
        \pi_1\left(\bigslant{T}{G}\right)\simeq \bigslant{\Lambda}{(\Lambda\cap\Gamma_{\fix})}\times_{\overline{\beta}}\: \bigslant{G}{G_{\fix}}.
    \]
    \begin{enumerate}
        \item Let $H:=G/G_{\fix}$ and $s\colon H \to G$ be a section. 
        Then every $g\in G$ can be written uniquely as $g=g_{\fix}\cdot s(h)$.
        \item For $g_{\fix}$, choose a lift $\hat{\tau}(g_{\fix})\in\CC^n$ of $\tau(g_{\fix})\in T$ such that the corresponding affine transformation
        \[
            \gamma_{\fix}(z):=\rho(g_{\fix})z+\hat{\tau}(g_{\fix})
        \]
        belongs to $\Gamma_{\fix}$. 
        For $s(h)$, choose an arbitrary lift $\hat{\tau}(s(h))\in\CC^n$. Then
        \[
            \hat{\tau}\colon G\longrightarrow \CC^n,\quad g_{\fix}\cdot s(h)\mapsto \rho(g_{\fix})\hat{\tau}(s(h))+\hat{\tau}(g_{\fix})
        \]
        is a lift of $\tau\colon G\to T$.
        \item The image of $[\tau]$ under $\delta$ is represented by
        \[
            \beta(g_1,g_2):=\rho(g_1)\hat{\tau}(g_2)-\hat{\tau}(g_1g_2)+\hat{\tau}(g_1)\in\Lambda.
        \]
        \item The particular choice of $\hat{\tau}$ guarantees that the cocycle $\beta$ descends to a cocycle $\overline{\beta}$ describing the fundamental group (cf.~\cite{charlap}*{Section~V.3}).
    \end{enumerate} 
\end{rem}


\section{Crepant Terminalizations and Resolutions of Singularities}\label{sec:resolution}

In this section, we construct crepant terminalizations and resolutions of our quotients and compare their deformation theory with the one of the singular quotients. More precisely, we show:  

\begin{prop}\label{prop:Resolutions}
	All quotients $Y$ in Theorem \ref{theo:main} admit a crepant terminalization $Y_{\ter}$ and a resolution $\hat{Y}$ fitting in the diagram
    \[\begin{tikzcd}
	 	Y_{\ter} \arrow{r}{\psi} & Y\\
	 	\hat{Y} \arrow{u}{\eta}\arrow{ur} &
	 \end{tikzcd}\]
    such that $Y_{\ter}$ and $\hat{Y}$ are infinitesimally rigid.
\end{prop}

To prove the proposition, we proceed as follows: first, we construct a crepant terminalization $\psi$ with the properties
\begin{align}\label{eq:PropertiesRes}
		\psi_\ast(\Theta_{Y_{\ter}})\simeq\Theta_Y\qquad \mathrm{and}\qquad		R^1\psi_\ast(\Theta_{Y_{\ter}})=0.
\end{align}
Leray's spectral sequence then yields an isomorphism $H^1(Y_{\ter},\Theta_{Y_{\ter}})\simeq H^1(Y,\Theta_Y)$. Thus, the rigidity of $Y_{\ter}$ follows from the rigidity of $Y$. It will turn out that the terminalizations have only cyclic quotient singularities of type $\tfrac{1}{d}(1,1,d-1)$, where $d=2,3,4$ or $6$. For varieties with such singularities, a resolution $\hat{Y}$ with the properties~\ref{eq:PropertiesRes} exists (cf. \cite{Kod1}), which implies that the first cohomology of $\Theta_{\hat{Y}}$ is trivial as well.



Since the quotients $Y$ have only isolated singularities, the construction of a suitable terminalization is a local problem. By Theorem~\ref{theo:Singularities}, the non-terminal singularities of $Y$ are all cyclic and of the following types:
\[
    \tfrac{1}{3}(1,1,1),\quad \tfrac{1}{7}(1,2,4), \quad \tfrac{1}{9}(1,4,7),\quad \tfrac{1}{14}(1,9,11).
\]

Recall that cyclic quotient singularities are toric. 
In \cite{BGKod1}, the authors give a crepant toric resolution of the singularities of type $\tfrac{1}{3}(1,1,1)$ having the properties~\ref{eq:PropertiesRes}. With the same methods, one can show that the crepant resolution of $\frac{1}{7}(1,2,4)$ given in \cite{RoanYau} enjoys these conditions as well. It remains to prove the existence of a suitable crepant terminalization for the last two singularities. Note that these singularities do not admit a crepant resolution.\\
As affine toric varieties, they are represented as follows: let $\sigma:=\cone(e_1,e_2,e_3)$ and consider the lattices
	\begin{align*}
		N_1&:=\ZZ^3+\ZZ\cdot \tfrac{1}{9}(1,4,7),\\
		N_2&:=\ZZ^3+\ZZ\cdot \tfrac{1}{14}(1,9,11)
	\end{align*}
in $\RR^3$. The affine toric variety representing the cyclic quotient singularity of type $\tfrac{1}{9}(1,4,7)$ is then given by $U_1:=\Spec(\CC[N_1^\vee\cap\sigma^\vee])$, where
\[
    N_1^\vee=\{ x\in\RR^3\mid \langle x,n\rangle \in\ZZ \quad \makebox{for all}\quad n\in N_1\} \quad \makebox{and}\quad \sigma^\vee=\{x\in \RR^3\mid \langle x, v\rangle \geq 0 \quad\makebox{for all}\quad  v\in\sigma\},
\]

are the dual lattice and the dual cone. The variety corresponding to $ \tfrac{1}{14}(1,9,11)$ is similarly given by $U_2:=\Spec(\CC[N_2^\vee\cap\sigma^\vee])$.\\
Subdividing the cone $\sigma$ along the rays generated by 
\begin{itemize}
	\item $v:=\tfrac{1}{3}(1,1,1)\in N_1$ or
	\item $v_1:=\tfrac{1}{7}(1,2,4),\:v_2:=\tfrac{1}{7}(4,1,2),\: v_3:=\tfrac{1}{7}(2,4,1)\in N_2$, respectively,
\end{itemize}
yields the fans $\Sigma_1$ and $\Sigma_2$ visualized in Figure~\ref{Fans}. 

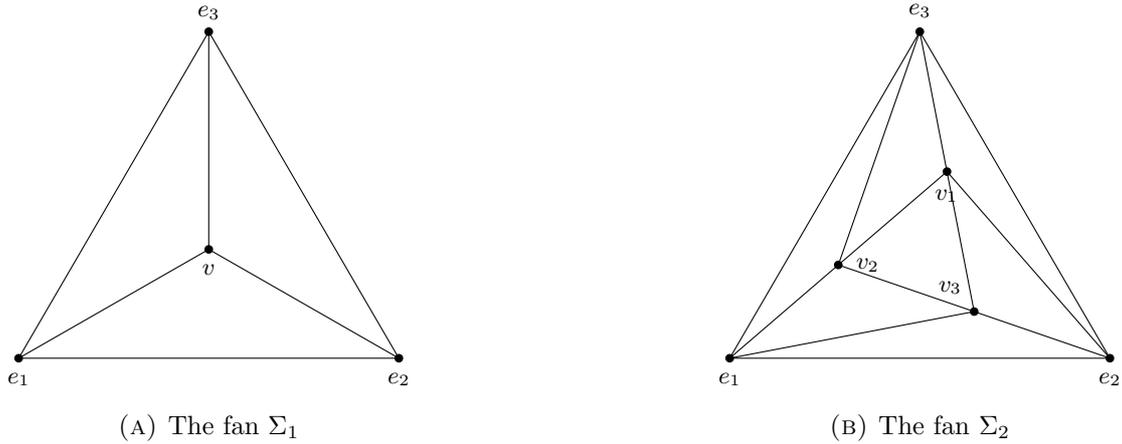
\begin{figure}[htb]
	\centering
	\begin{minipage}[t]{.45\linewidth}
		\centering
		\begin{tikzpicture}
			\def\sl{5} 
			\def\w{60} 
			\def\dotsize{0.01*\sl}
			\def\name{\tau}
			\def\size{\footnotesize}
			
			\coordinate[label={[label distance=0.06 cm]-90:\size $e_1$}] (A) at (0,0);
			\coordinate[label={[label distance=0.06 cm]-90:\size $e_2$}] (B) at (\sl,0);
			\coordinate[label={[label distance=0.03 cm]90:\size $e_3$}] (C) at (\w:\sl);
			
			\coordinate (fc) at ($(A)!0.5!(B)$);
			\coordinate (fa) at ($(B)!0.5!(C)$);
			\coordinate (fb) at ($(C)!0.5!(A)$);
			
			\coordinate[label={[label distance=0.06 cm]-90:\size $v$}] (v) at (intersection of A--fa and B--fb);
			
			\draw (A) -- (B) -- (C) -- cycle;
			\draw (A) -- (v);
			\draw (B) -- (v);
			\draw (C) -- (v);
			\filldraw (A) circle (\dotsize);
			\filldraw (B) circle (\dotsize);
			\filldraw (C) circle (\dotsize);
			\filldraw (v) circle (\dotsize);
			
		\end{tikzpicture}
		\subcaption{The fan $\Sigma_1$}
	\end{minipage}%
	\hfill%
	\begin{minipage}[t]{.45\linewidth}
		\centering
		\begin{tikzpicture}
			\def\sl{5} 
			\def\w{60} 
			\def\dotsize{0.01*\sl}
			\def\name{\tau}
			\def\size{\footnotesize}
			
			\coordinate[label={[label distance=0.06 cm]-90:\size $e_1$}] (A) at (0,0); 
			\coordinate[label={[label distance=0.06 cm]-90:\size $e_2$}] (B) at (\sl,0);
			\coordinate[label={[label distance=0.06 cm]90:\size $e_3$}] (C) at (\w:\sl);
			
			\coordinate (fc) at ($(A)!0.6667!(B)$); 
			\coordinate (fa) at ($(B)!0.6667!(C)$);
			\coordinate (fb) at ($(C)!0.6667!(A)$);
			
			\coordinate[label={[label distance=0.1 cm]-93:\size $v_1$}] (v1) at (intersection of A--fa and C--fc);
			\coordinate[label={[label distance=0.1 cm]105:\size $v_3$}] (v3) at (intersection of B--fb and C--fc); 
			\coordinate[label={[label distance=0.1 cm]0:\size $v_2$}] (v2) at (intersection of A--fa and B--fb);
			
			\coordinate (fd) at ($(A)!0.5!(v3)$); 
			\coordinate (fe) at ($(B)!0.5!(v1)$);
			\coordinate (ff) at ($(C)!0.5!(v2)$);
			
			\coordinate (fg) at ($(v2)!0.5!(v3)$);
			\coordinate (fh) at ($(v1)!0.5!(v3)$);
			
			\draw (A) -- (B) -- (C) -- cycle;
			
			\draw (A) -- (v1);
			\draw (B) -- (v2);
			\draw (C) -- (v3);
			
			\draw (A) -- (v3);
			\draw (B) -- (v1);
			\draw (C) -- (v2);
			
			\filldraw (A) circle (\dotsize);
			\filldraw (B) circle (\dotsize);
			\filldraw (C) circle (\dotsize);
			\filldraw (v1) circle (\dotsize);
			\filldraw (v2) circle (\dotsize);
			\filldraw (v3) circle (\dotsize);
			
			
			
		\end{tikzpicture}
		\subcaption{The fan $\Sigma_2$}
	\end{minipage}
	\caption{Terminalizations of the singularities of type $\tfrac{1}{9}(1,4,7)$ and $\tfrac{1}{14}(1,9,11)$}
    \label{Fans}
\end{figure}

The affine toric varieties corresponding to the maximal cones of the fan $\Sigma_1$ are  cyclic quotient singularities of type $\tfrac{1}{3}(1,1,2)$ and the ones corresponding to the maximal cones of $\Sigma_2$ are of type $\tfrac{1}{2}(1,1,1)$, hence  the corresponding toric varieties $X_{\Sigma_1}$ and $X_{\Sigma_2}$ have only terminal singularities. Thus, we obtain toric terminalizations
$$\psi_j\colon X_{\Sigma_j}\to U_j.$$
Since the vectors $v,v_1,v_2$ and $v_3$ belong to the plane
\[
    \{ (x_1,x_2,x_3)\in \RR^3\mid x_1+x_2+x_3=1\},
\]
these terminalizations are crepant by the Reid-Shepherd-Barron-Tai criterion (cf.~\cite{Reid87}).


It finally remains to verify that the terminalizations $\psi_j\colon X_{\Sigma_j}\to U_j$ actually satisfy the conditions~\ref{eq:PropertiesRes}.

\begin{Notation}
	We denote by $D_i'\subset U_j$ and $D_i\subset X_{\Sigma_j}$ the divisors corresponding to the rays generated by $e_i$, $i=1,2,3$, and let $E_w\subset X_{\Sigma_j}$ be the exceptional divisor of the terminalizations corresponding to the added ray generated by $w$, where $w=v$ if $j=1$, and $w\in\{v_1,v_2,v_3\}$ else.
\end{Notation}

\begin{Lemma}
	The terminalizations $\psi_j$, $j=1,2$, fulfill $(\psi_j)_\ast(\Theta_{X_{\Sigma_j}})\simeq \Theta_{U_j}$.
\end{Lemma}

\begin{proof}
	We give a proof for the case $j=2$, the other one is similar. 
	By \cite{BGKod1}*{Proposition~5.8}, which even holds for partial toric resolutions, we have to show that $P_{D_i}\cap N_2^\vee=P_{D_i'}\cap N_2^\vee$ holds for all $i=1,2,3$. By symmetry, it is enough to consider the case $i=1$. The polyhedrons of the divisors are given by
	\begin{align*}
		P_{D_1'}&=\{x\in\RR^3\mid x_1\geq -1,\quad x_2,\, x_3\geq 0\},\\
		P_{D_1}&=P_{D_1'}\cap\{x\in\RR^3\mid \langle x,v_k\rangle \geq 0,\: k=1,2,3\}.		
	\end{align*}
	Thus, the inclusion $P_{D_1}\cap N_2^\vee\subset P_{D_1'}\cap N_2^\vee$ is obvious. For the converse, let $x\in P_{D_1'}\cap N_2^\vee\subset\ZZ^3$. Then, since $v_1\in N_2$, it holds
	\[
	-1 \leq x_1+2x_2+4x_3 \equiv 0 \mod 7.
	\]
	This implies that the sum $x_1+2x_2+4x_3 =7\cdot\langle x, v_1\rangle$ is non-negative. Analogously, $\langle x,v_k\rangle \geq 0$ for $k=2,3$. Hence, $x$ belongs to $P_{D_1}$.
\end{proof}

\begin{Lemma}
	The terminalizations $\psi_j$ fulfill $R^1(\psi_j)_\ast(\Theta_{X_{\Sigma_j}})=0$.
\end{Lemma}

\begin{proof}
	We only verify the assertion in the case $j=2$. For simplicity, we drop the index $j$ in the following.\\
    The dual of toric Euler-sequence (cf. \cite{CLS11}*{Theorem~8.1.6}) on $X_\Sigma$ reads
    \begin{align}\label{eq:EulerDual}
	   0\longrightarrow \sO_{X_\Sigma}^{\oplus r}\longrightarrow  \bigoplus_{i=1}^3\sO_{X_\Sigma}(D_i)\oplus\bigoplus_{k=1}^3\sO_{X_\Sigma}(E_k)\longrightarrow \Theta_{X_\Sigma}\longrightarrow 0.
    \end{align}
    Since $U$ and $X_\Sigma$ have rational singularities, this yields an isomorphism
\begin{align}\label{eq:R1}
	R^1\psi_\ast\Theta_{X_{\Sigma}}\simeq \bigoplus_{i=1}^3R^1\psi_\ast\sO_{X_\Sigma}(D_i)\oplus\bigoplus_{k=1}^3R^1\psi_\ast\sO_{X_\Sigma}(E_k).
\end{align}
  Thus, since $U$ is affine,  it is enough to prove:
    	\begin{enumerate}[(1)]
		\item $H^1(X_\Sigma, \sO_{X_\Sigma}(D_i))=0$ for $i=1,2,3$
		\item $H^1(X_\Sigma, \sO_{X_\Sigma}(E_{v_k}))=0$ for $k=1,2,3$.
	\end{enumerate}
     
	By symmetry, it is enough to consider the cases $i=1$ and $k=1$. Since $X_\Sigma$ has only singularities of type $\tfrac{1}{2}(1,1,1)$, the divisor $D_1$ is $\QQ$-Cartier with index $2$.  The Cartier data for $2D_1$ are given by the elements $m_{\tau_2}=(-2,\ 8,\ 0),\:m_{\tau_3}=m_{\tau_{23}}=(-2,\ 0,\ 4)$, where $\tau_2=\cone(e_1,e_3,v_2),\,\tau_{3}=\cone(e_1,e_2,v_3)$ and $\tau_{23}=\cone(e_1,v_2,v_3)$, and $m_{\tau}=0$ for all other maximal cones $\tau$ of $\Sigma$. Since all these elements belong to the polyhedron associated to $2D_1$,
    \[
        P_{2D_1}=\{x\in\RR^3\mid x_1\geq -2,\quad x_2,\, x_3\geq 0, \quad \langle x,v_k\rangle \geq 0,\: k=1,2,3\},
    \]
    we conclude that the Cartier divisor $2D_1$ is basepoint free (cf.~\cite{CLS11}*{Proposition~6.1.1}), hence nef.  The vanishing of the cohomology group of the divisor $D_1$ follows now from the theorem of Demazure (cf. \cite{CLS11}*{Theorem~9.2.3}).
	
	Finally, let $E:=E_{v_1}=\sum_{\rho\in\Sigma(1)}a_\rho D_\rho$, where $a_\rho=1$ if $\rho=\cone(v_1)$ and $a_\rho=0$ else. Since $E$ is not nef, we can not apply Demazure-vanishing. Instead, we show that the sets
    $$V_{E,m}:=\bigcup_{\tau\in \Sigma_{\max}}\Conv(u_\rho\mid \rho\in\tau(1),\langle m,u_\rho\rangle <-a_\rho)$$
    are connected for all $m\in N^\vee$, where the sum runs over all maximal cones of the fan $\Sigma$. This implies that $H^1(X_\Sigma, \sO_{X_\Sigma}(E)))=0$ (cf. \cite{CLS11}*{Theorem~9.1.3}).
    
    Looking at the illustration of the fan $\Sigma=\Sigma_2$ in Figure~\ref{Fans}, we see that $V_{E,m}$ is disconnected if  and only if 
    there exists an $i$ such that $e_i,v_i\in V_{E,m}$, but $e_j, v_j \notin V_{E,m}$ for all $j\neq i$. We only treat the case $i=1$, as the other cases are analogous. The conditions $v_1\in V_{E,m}$ and $v_2\notin V_{E,m}$ are equivalent to $\langle m,v_1\rangle < -1$ and $\langle m,v_2\rangle \geq 0$, written out:
    \[
        m_1+2m_2+4m_3 < -7 \qquad \text{and}\qquad  -4m_1-m_2-2m_3\leq 0 .
    \]
    By adding the first inequality to two times the second inequality, we obtain
    $-7 m_1< -7$, so $m_1> 1$. This is a contradiction because the condition $e_1\in V_{E,m}$ means that $m_1=\langle m,e_1\rangle <0$. 
\end{proof}


\section{Proof of the Main Theorem~\ref{theo:main}}\label{sec:proofMain}

In the previous sections, we discussed in several steps all that we need for the proof of the main Theorem~\ref{theo:main}. In this final section, we summarize the reasoning.

\begin{proof}[Proof of Theorem~\ref{theo:main}]
    Let $G$ be a finite group admitting a rigid, holomorphic and translation-free action on a 3-dimensional complex torus $T$  with finite fixed locus such that the quotient has canonical singularities and $p_g=0$. 
By Theorem~\ref{theo:groups}, $G$ must be isomorphic to one of the following groups: $\ZZ_9$, $\ZZ_{14}$, $\ZZ_3^2$, $\ZZ_3^3$, or $\ZZ_9\rtimes\ZZ_3$.
Furthermore, as stated in Proposition~\ref{prop:ConsBieb}, the quotients obtained by different groups cannot be homeomorphic, and the  homeomorphism and diffeomorphism classes of the quotients are the same.

    For the cyclic groups, the classification is easy: there is one and only one biholomorphism class (cf. Proposition~\ref{prop:Cyclic}).\\
    In the cases $G=\ZZ_3^2$, $G=\ZZ_3^3$ and $G=\ZZ_9\rtimes\ZZ_3$, the situation is more involved. Proposition~\ref{prop:AnalyticRepr} shows, that there are two possibilities for the analytic representation for $\ZZ_3^2$, whereas the representations of $\ZZ_3^3$ and $\ZZ_9\rtimes \ZZ_3$ is unique up to equivalence of representations and automorphisms of the groups. In any case, we can deduce the structure of the torus from the description of the linear part of the action: it is the quotient of three copies of Fermat's elliptic curve $E$ by a subgroup $K$ of $E[3]^3$, which is the kernel of an isogeny given by addition. 
    
    If $G=\ZZ_3^3$, then there are four candidates for the kernels but only two of them allow actions with isolated fixed points (Remark~\ref{rem:ActionsZ3^3}). The classification is settled in Proposition~\ref{prop:ClassificationZ3^3}.\\
    If $G=\ZZ_3^2$, we have two subcases according to the two choices $\rho_1$ and $\rho_2$ of the analytic representation. They lead to distinct $\mathcal C^{\infty}$-classes of  quotients because  the representations in the 
    $\Aut(\mathbb Z_3^2)$-orbit of $\rho_1$ are not  equivalent to $\rho_2$, even considered as real representations.    The classification of the quotients where the action has linear part $\rho_1$ is summarized in Proposition~\ref{prop:oneBihol}. If the analytic representation equals $\rho_2$, then two kernels are possible and the fine classification is explained in Proposition~\ref{prop:Classification2Z3^2}.\\ 
    In the case $G=\ZZ_9\rtimes \ZZ_3$, there is one and only one biholomorphism class. The proof is given in Proposition~\ref{prop:ClassificationZ9:Z3}.\\
    Next, we explain how to determine the baskets of singularities in each case. In Theorem~\ref{theo:Singularities}, we prove that all stabilizer groups are cyclic and determine the possible types of canoncial singularities. The orbifold Riemann-Roch formula (Proposition~\ref{prop:CondNumberSing}) allows us to compute the possible baskets of all non-Gorenstein singularities.\\
    The Propositions~\ref{prop:-id}, \ref{prop:ze3} and \ref{prop:LastCase} ensure that only the cases $k=9$ ($G=\ZZ_{14}$), $k=12$ ($G=\ZZ_{9}$ or $\ZZ_9\rtimes \ZZ_3$) and $k=15$ ($G=\ZZ_3^2$ or $\ZZ_3^3$) of Corollary~\ref{cor:BasketsSing} can occur. To count the Gorenstein singularities, we use the Lefschetz fixed-point formula (Lemma~\ref{le:Fix} and \ref{le:Lefschetz}).\\    
    In the Corollaries~\ref{cor:universalCover} and \ref{cor:fundamentalgroup} and Remark~\ref{rem:fundamentalgroup}, the fundamental groups of the quotients and the structure of their universal covers are described.\\
    Finally, Section~\ref{sec:resolution} shows that each quotient admit an infinitesimally rigid crepant terminalization with numerically trivial canonical divisor and furthermore, smooth rigid 3-folds as resolutions.
\end{proof}


\end{document}